\documentclass[12pt,leqno]{article} 
\usepackage{amsmath}
\usepackage{amssymb}
\usepackage{amsthm} 
\usepackage{epsf}
\usepackage{graphicx}
\usepackage{esint} 
\usepackage{dsfont}
\usepackage{hyperref} 
\hypersetup{pdfpagemode=FullScreen,  colorlinks=true} 
\usepackage{enumerate} 

\numberwithin{equation}{section}

\newtheorem{lem}{Lemma}[section] 

\newtheorem{defi}[lem]{Definition}
\newtheorem{thm}[lem]{Theorem}

\theoremstyle{remark}
\newtheorem{rem}[lem]{\bf Remark}
\newtheorem{qn}[lem]{\bf Question}

\newcommand{\nn}{\nonumber}
\newcommand{\ms}{\medskip}
\newcommand{\msi}{\par\medskip\noindent}

\newcommand{\R}{\mathbb{R}}

\renewcommand{\H}{\mathcal H}

\newcommand{\1}{\mathds{1}}

\newcommand{\bZ}{\mathbb {Z}}

\renewcommand{\d}{\partial}

\newcommand{\sm}{\setminus}
\newcommand{\supp}{\mathrm{supp}}

\newcommand{\wt}{\widetilde}
\newcommand{\wh}{\widehat}
\newcommand{\ol}{\overline}

\newcommand{\cE}{{\cal E}}

\newcommand\Bk{\color{black}}

\begin{document}

\title{On a control problem for two orthogonal nets of curves in the plane}

 \author{Guy David 
 \thanks{GD has been partially supported by the Simons Foundation through through the Simons Collaboration on Localisation of Waves, 601941 GD, and the Simons Initiative on the Geometry of Flows (Grant Award ID BD-Targeted-00017375 GD).} \, and 
 Polina Perstneva \thanks{PP has been partially supported by Fondation Math\'ematique Jaques Hadamard, the program L'Or\'eal-UNESCO 
 Pour les Femmes et la Science, and Swiss NSF.}} 
 
 \date{\today}

\maketitle

\begin{abstract}
We construct nets of curves in a rectangle $R$ with special properties. The net is formed by two families of curves, ``horizontal'' and ``vertical'', each family foliates $R$, and every horizontal curve meets every vertical curve orthogonally. Moreover, every horizontal curve started at the left edge of $R$ arrives to the right edge of $R$ at the same height, and every vertical curve started at the bottom edge of $R$ arrives to the top edge of $R$ at a prescribed in advance longitude. The latter longitude function as a map from bottom to top can be chosen in advance as any small perturbation of the identity.

Our motivation for this was to design some elliptic divergence form
operators, for instance on one of the domains bounded by a Reifenberg flat curve in the plane, with the scalar coefficient and the corresponding elliptic measure being absolutely continuous with respect to the natural measure on the boundary (an object of a separate paper). 

The construction of our net of curves in a rectangle can also be interpreted as control problem, has an interest on its own, and, likely, has other applications.

~\ 

\end{abstract} 

\tableofcontents

\section{Introduction}
\label{S1}

The main goal of this article is to construct (a net of) two mutually orthogonal families of curves in a rectangle, with the properties precisely described below in Subsection \ref{S1S1}. The motivation for the problem is presented in Subsection \ref{S1S3}, after a brief discussion directly relevant to our main Theorem \ref{t1}. This discussion is in a form of multiple remarks, and is localised in a short Subsection \ref{S1S2}.

\subsection{Orthogonal nets of curves and the main statement} 
\label{S1S1}

We are given a rectangle $R = I \times J \in \R^2$ 
and a function $f : I \to I$, a small perturbation of the identity,
and we would like to construct a foliation of $R$ by green curves $G_t$ (roughly horizontal), 
and a foliation of $R$ by red curves $R_x$ (roughly vertical), so that the red curves always cross the green ones orthogonally, all the curves are smooth enough,
and in addition, writing $I = [a,b]$ and $J = [c,d]$, 
\begin{equation}\label{1a1}
\text{ For $t\in J$, the green curve $G_t$ that starts at $(a, t)$ ends at $(b,t)$}
\end{equation}
and 
\begin{equation}\label{1a2}
\text{ For $x\in I$, the red curve $R_x$ that starts at $(x, c)$ ends at $(f(x),d)$.}
\end{equation}
We also require (naturally) that $\d R$ is composed of red and green curves, namely
\begin{equation}\label{1a3}
G_c = I \times \{ c \}, \  G_d = I \times \{ d \},  \ R_a = \{ a \} \times J,  \  \ R_b = \{ b \} \times J.
\end{equation}
See Figure \ref{fig1}. 
\begin{figure}[!h]
  \centering 
 \includegraphics[width=9cm]{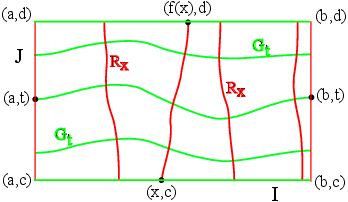} 
  \caption{Our requirements on green and red curves}
\label{fig1}
\end{figure}

\medskip

To be more specific about the assumptions on $f$, we extend it to $\R$ by taking $f(x) = x$ for $x \in \R \sm I$,
and then define $\psi : \R \to \R$ by $\psi(x) = f(x) - x$ for $x\in \R$. We require that for some integer
$k \geq 2$ 
\begin{equation}\label{1a4}
\psi : \R \to \R \text{ is supported in $I$ and of class $C^k$ with } ||\psi||_k \leq \varepsilon
\end{equation}
for some small enough $\varepsilon > 0$. Here we use the homogeneous norm
\begin{equation}\label{1a5}
N_k(\psi) = ||\psi||_k 
= \sup_{x \in \R}  |\psi^{(k)}(x)|.
\end{equation}
Notice, however, that since $\psi$ is supported on $I$, we also have control over the previous derivatives, namely
\begin{equation}\label{1a6}
N_{j}(\psi) \leq (|I|/2)^{k-j} N_k(\psi) \ \text{ for } 0 \leq j < k.
\end{equation}
Also notice that, for $\varepsilon$ small, $f$ is an increasing bijection of $I$.

\medskip

Before we state our main result, we 
also need to make much more explicit what we mean by foliations of $R$ by roughly
horizontal (or vertical) curves. We will require every green curve $G_t$, $t \in J$,
to be the (intersection with $R$ of the) graph of some $C^m$ function $g_t : \R \to \R$, such
that $g_t(x) = t$ for $x \in \R \sm I$ (so that we don't have potential gluing problems along the vertical edges of $R$).
All the graphs $G_t$ that we construct will be $10^{-1}$-Lipschitz, say (i.e., $N_1(g_t) \leq 10^{-1}$).
We may also extend the curves and define $g_t(x) \equiv x$ for $t \notin J$.

Similarly, each red curve $R_x$, $x \in I$, is the graph over the vertical axis (i.e., with the coordinates exchanged) of a 
$10^{-1}$-Lipschitz and $C^m$ function $r_x : \R \to \R$, with $r_x(t) = x$ for $t\leq c$ and 
$r_x(t) = f(x)$ for $t \geq d$. Again, we can extend the foliation beyond $R$ and define $r_x(t) = t$ for $x \notin I$.

We complete these assumptions with a strong continuity constraint on the maps $t \mapsto g_t$ and $x \mapsto r_x$
that we take as our definition of the ``foliation'' constraint. We want the functions $(x,t) \mapsto g_t(x)$
and $(x,t) \mapsto r_x(t)$ to be of class $C^m$ (this looks too strong, but we will
get it easily) and that each point $(X,T)$ of $R$ lies in a single $G_t$ and a single $R_x$. 
Strictly speaking, this means that there is only one $t$ such that $T = g_t(X)$ and only one $x$ such that $X = r_x(T)$.

This allows us to define a nearly horizontal vector field $(X,T) \mapsto e(X,T)$ tangent at $(X,T)$ to the unique $G_t$ that goes through $(X,T)$, and points to the right, namely
\begin{equation} \label{1c7}
e(X,T) = (1+g_t'(X)^2)^{-1/2} \big(1, g_t'(X)\big), \text{ where $t$ is such that $T = g_t(X)$,}
\end{equation}
and similarly a nearly vertical vector field $n(X,T)$ tangent at $(X,T)$ to the unique $R_x$ that
goes through $(X,T)$, and points upwards, namely
\begin{equation} \label{1c8}
n(X,T) = (1+r_x'(T)^2)^{-1/2} \big(r_x'(T), 1\big), \text{ where $x$ is such that $X = r_x(T)$.}
\end{equation}
This completes the foliation part of the definition.

Then there is the orthogonality condition that $G_t$ is orthogonal to $R_x$ at their point of intersection
(it is unique, directly by the small graph descriptions). In terms of our vector fields, this means that 
\begin{equation} \label{1a7}
e(X,T) \perp n(X,T)  \ \text{ for } (X,T) \in R.
\end{equation}

We tried to make our definitions work also when $m = 1$, in which case $e$ and $n$ are only expected to be continuous 
(because the mapping $(X,T) \mapsto (x,t)$ is continuous, as the inverse of a continuous bijective mapping
$(x,t) \mapsto G_t \cap R_x$ on the compact set $R$). 
But as soon as $m \geq 2$, our proof shows that $e$ and $n$ are of class $C^{m-1}$, and then we could have
used the Lipschitz vector fields to define the green and red curves as integral curves of $e$ and $n$, thus stating the theorem
below only in terms of vector fields. We find it interesting to be able to think about both descriptions, but will not try
to translate precisely the assumptions on the curves into assumptions on the vector fields.

\medskip

We are now ready to state our main result about the existence of a foliation of a rectangle $R$ as described above, with properties
\eqref{1a1} - \eqref{1a3} and \eqref{1a6} - \eqref{1a7}), 
where the function $f$ satisfies \eqref{1a4}.

\begin{thm}\label{t1} 
Let $R = [0,1] \times J \in \R^2$, 
with $|J| \leq 1$, 
$k \geq 2$, and let $m$ be the integer part of $k/2$. 
Let $f : \R \to \R$, $\psi(x) = f(x) - x$ be as above, and suppose that 
\eqref{1a4} holds for some small enough $\varepsilon = \varepsilon(k)$. 
Then we can find two orthogonal (in the sense of \eqref{1a7}) foliations of $R$ by curves $G_t$, $R_x$, of class $C^{m}$, 
with the properties \eqref{1a1}, \eqref{1a2}, and \eqref{1a3}.
The term foliation is explained just above, and we even get that the mappings
$(x,t) \mapsto g_t(x), r_x(t)$ are $C^{m}$ in both variables.
Furthermore, the $C^{m}$ norm of these functions 
is less than  $C |J|^{-m-1} N_k(\psi)^{1/2}$, 
where $C$ does not depend 
on $|J|$. 
\end{thm}

Our bound in $|J|$ is probably not optimal. Also, in our proof $\varepsilon(k)$ 
really depends on $|J|$; 
see for instance the assumptions for Lemma \ref{l1a5}. Note here also that 
the assumption $|I| \ge |J|$ may be dropped as long as one cares about constructing the foliations with curves, 
but does not worry so much about the bounds.

For our application to elliptic operators, we want our curves to be smooth enough (for instance, $m \geq 2$),
so we need to take $k\geq 4$, say, but in fact it would not disturb us to require a much smoother $f$ (and with a very small norm too). 
We do not claim that our conditions on $m$ and $k$ are optimal, and we lose some derivatives 
(more than we expected: see Remark~\ref{r1z5}), as we also lose on the smallness of the norms.

\subsection{Some remarks and questions}
\label{S1S2}

\begin{rem} [\bf About some easy extensions on the plane]  \label{r1z2} 
We stated Theorem \ref{t1} for small data $f$, but Lemma \ref{l1a4} below gives a solution
for the composition $f \circ g$ as soon as we have solutions for $f$ and for $g$.
This means that Theorem \ref{t1} immediately extends to increasing $C^k$
diffeoeomorphisms of $I$, because they can be decomposed into a finite composition of small mappings $f_k$
(i.e., with $N_k(f_k - Id)$ as small as we want) that satisfy \eqref{1a4}. 
Of course the number of pieces, and then the $C^m$ constants for each curve, diverge, 
because we are forced to take very thin intervals $J_k$.

\medskip
By dilation invariance, Theorem \ref{t1} stays true when we multiply the lengths of $I$ and $J$ by the same constant $\lambda > 0$. Then $\varepsilon(k)$ stays the same (depending only on $|I|/|J|)$, but our bounds for the $C^{m}$ norm become different
(to account for dilations).

In addition, we can get a bound in Theorem \ref{t1} with $|I|$ larger than $[0, 1]$ without using the dilatation (and consequently without having to add the parameter $\lambda$ to the norm estimate), but 
provided that in addition to \eqref{1a4}, we require that 
\begin{equation} \label{1b8}
N_0(\psi) : = 
||\psi||_\infty \leq \varepsilon.
\end{equation}
And the same thing is true on the strip $R = \R \times J$, with any $C^k$ function $f$ such that
\begin{equation} \label{1b9}
||\psi||_\infty + ||\psi^{(k)}||_\infty 
\leq \varepsilon,
\end{equation}
where we drop the requirement \eqref{1a1} (but still require the $G_t$ and $R_x$ to foliate $R$). We do this by applying our curtain trick from Section \ref{S3}: it allows us to treat separately some (potentially infinite) number of disjoint horizontal intervals of length not greater than one (and apply to them separately Theorem \ref{t1} with the restrictions $I = [0, 1]$ and $|J| < 1$).
\end{rem}

\begin{rem} [\bf About the problem being underdetermined]
The fact that the question of Theorem \ref{t1} is brutally underdetermined partially accounts for its difficulty. For a given $f$ there should be lots of different nets of curves that solve the problem; the complication is to find reasonable ones, that we construct almost explicitly. More precisely, we shall try to solve the problem by writing $f$ as a composition of simpler mappings, for which we have specific solutions.
\end{rem}

\begin{rem}[\bf About control] \label{r1a2} 
We find it interesting to see Theorem \ref{t1} as a control result, 
where we need to find the smooth vector field $e(x,t)$ 
(if $m \geq 2$,  this is equivalent to finding the green and red curves), so that the outputs of two transformations,
namely following the vector field $e$ to go from $(a,t)$ on the left of $R$ to a point $(b, g(t))$ on the right,
and following the orthogonal vector field $n$ to go from $(c,x)$ on the bottom of $R$ to a point $(d, f(x))$
on the top, are given ones, namely the identity (for the horizontal crossing following the green curves) 
and $f$ (for the vertical crossing following the red curves). 
See Remark \ref{r1a4} and Section \ref{S11} 
below concerning the more general case of two independent outputs $g$ and $f$.

\medskip
It is very tempting to use our smallness assumption \eqref{1a4} to proceed by linearization.
That is, we would look for vector fields $e(x, t) = e_1 + v$, $v$ small, linearize the outputs $O_1$ 
and $O_2$, find an inverse for them, and then try to iterate or apply a fixed point theorem. 

We did not manage this directly, and see two difficulties with this approach. 
The first one is that the range of the linearization is a strict subspace, i.e.,
is composed of functions $f = Id + \psi$ for which $\int_I  \psi(x) = 0$. 
This does not mean that the other functions $f$ cannot be attained 
(and we will see some examples in Section \ref{S5}), 
but we claim that they cannot be attained by vector fields of order $1$.

The second difficulty is that our constructions to fix this problem and iterate anyway seem all to 
eat one derivative of $v$ (each time we iterate!), which is not good. 
This appears when we want to control the errors that we make when we approximate
$O_1$ and $O_2$ by the linear parts; a similar difficulty will appear in our proof below, 
late in Section \ref{S8}, but then, luckily, we will differentiate a function which is a wavelet, whose derivative has roughly 
the same size as the function itself.
\end{rem}

\begin{rem} [\bf About the size of the solution]  \label{r1a3}
The curves constructed in Theorem \ref{t1} have $C^{m}$ norms less than $C N_k(\psi)^{1/2}$, 
or equivalently the $C^{m}$ norm of $e(x,t)$ is less than $C N_k(\psi)^{1/2}$. This is not as surprising as it may seem: first, a direct linearization does not work, and the remark above says that a linear growth is most probably false. Second, at some point, in order to modify our targets a little, we will need to modify the curves much more. See some explanations at the end of Section \ref{S7}.
\end{rem}

\begin{rem} [\bf About the lost derivatives]  \label{r1z5}
It is more intriguing that we loose about half of the derivatives of $f$ (we initially expected to lose only $1$). 
This may be for purely technical reasons. In the proof, we compose infinitely many mappings 
(at smaller and smaller scales), and, because we need to compose with maps that are larger than the current errors
(as in Remark \ref{r1a3}), the corresponding curves that we construct are more and more distorted.
Again, see the end of Section \ref{S7} for the actual estimates. 
Naturally, there is a relation between the norms of the functions we compose and the regularity of the final mapping, 
which is why we pay so much in the end.

We were kindly asked by C. De Lellis whether there is a $C^\infty$ version of Theorem \ref{t1}.
If we don't want to use special classes of functions, this probably means finding an algorithm that would
yield $C^\infty$ solutions given small $C^\infty$ data $f$, but we do not immediately see how
smallness will be measured then, unless we have a different condition for each derivative. 
Certainly we would have to replace our compactly supported wavelets
by rapidly decreasing ones (or Calder\'on decompositions), and the cost would be for our approximation algorithm below to get more complicated. Probably we could choose $\tau$ in Section \ref{S9} independent of $k$, but we did not try to pursue this direction here.
\end{rem}

\begin{rem} [\bf A question about double control]  \label{r1a4}
We may want to replace \eqref{1a1} by the more general requirement that
\begin{equation}\label{1a8}
\text{ For $t\in J$, the green curve $G_t$ that starts at $(a, t)$ ends at $(b, g(t))$},
\end{equation}
where $g : J \to J$ is another increasing diffeomorphism. Thus Theorem \ref{t1} corresponds to $g(t) = t$. We do not know 
whether the analogue of Theorem \ref{t1} is still true in this context. Our proof does not extend, but if we give ourselves 
some room around $R$ for the curves to organize themselves (instead of forcing them to be the straight lines), we can use 
Theorem~\ref{t1} to get the result. See Section \ref{S11}. 
\end{rem}

\subsection{Motivation}
\label{S1S3}

The problem described in Subsection \ref{S1S1} seems interesting to us on its own, though we did not find anything similar in the literature. As explained in Remark \ref{r1a2}, it can be interpreted as a control problem. It also can be interpreted as a complex analysis problem: one tries to find a quasiconformal mapping (from a disk to itself or from a half-plane to itself) with prescribed boundary data and some condition on the coefficient of the mapping. 
Here we think about the bijective mapping on $R$ that sends a pair $(x,t)$ to the intersection of $G_t$ and $R_x$, 
and the main condition is the orthogonality \eqref{1a7} of the level sets of co-ordinates.

Our original motivation was different, and stems from the study of harmonic/elliptic measure on rough domains.
We give a very rapid overview here, and refer the reader to the companion paper \cite{DP} 
where we apply Theorem \ref{t1}.

Consider an elliptic operator $L = - \rm{div} A \nabla$ on a domain $\Omega \subset \R^n$, where in general
$A$ is a matrix-valued bounded measurable function satisfying some uniform ellipticity condition. 
The special case of $A \equiv Id$ and $L = -\Delta$ (the laplacian) is already very interesting. 
There has been a lot of efforts to understand the elliptic measure associated to $L$ in $\Omega$
(this is the measure on $\d\Omega$ that allows one to compute solutions of the Dirichlet problem $Lu=0$ with
$u_{\vert d\Omega}$ given). See introductions of \cite{ABHM}, \cite{AHMMT} and \cite{AHM3TV}, or the whole book 
\cite{PT} for more information about this, and in particular the question of whether the elliptic measure is absolutely
continuous with respect to the natural measure on $\d\Omega$ (often, the surface measure $\H^{n-1}_{\vert \d\Omega}$). This question becomes very important if one wants to determine whether the Dirichlet problem for the aforementioned operator $- \rm{div} A \nabla$ with rough boundary data ($L^p$) is solvable in this domain. 

Typical positive results would give that for one-sided NTA domains in $\R^n$, with $(n-1)$-dimensional
Ahlfors regular boundaries, and $L = - \Delta$, or more generally when $A$ is well approximated locally 
by constant coefficient operators (the so-called Dahlberg-Kenig-Pipher (DKP) condition), elliptic measure and surface 
measure are mutually absolutely continuous, with uniform estimates, as soon as $\d\Omega$ is uniformly rectifiable.
In the aforementioned papers, general enough converse results were obtained as well, to the effect that
(for DKP coefficients), rectifiability is needed. But it still makes sense to ask when, depending on the pair
$(\Omega, A)$, we have absolute continuity of one measure with respect to the other, 
knowing from the recent \cite{FP} that the answer stays the same for DKP perturbations of $A$.

The paper \cite{DP} follows two recent papers \cite{DM, P}, where it was shown that elliptic measure can be equivalent
to the natural (Hausdorff) measure on $\d\Omega$ for some fractal boundaries (thus very far from rectifiability), but with 
specifically chosen coefficients $A$ (which are therefore very far from a DKP perturbation of the identity coefficient). 
For all these counterexamples, 
the matrix $A$ is scalar, i.e., $A = a I$ for some $a: \Omega \to [C^{-1}, C] \subset \R$, which is natural 
for reasons explained in these papers. These examples are obtained by first choosing the Green function
(carefully), and then showing that it satisfies an elliptic equation. Incidentally, in the context of \cite{P} and \cite{DP}, 
producing examples with general matrix-valued $A$ is too easy, because we could use quasiconformal mappings.

In \cite{DM}, the main studied domain is the complement in the plane of the 
Garnett-Ivanov Cantor set of dimension $1$, which is probably the most celebrated Ahlfors regular example for which
the harmonic measure (for $L = -\Delta$) is singular. In \cite{P}, the examples were the complementary
components of Koch-type snowflakes in the plane. In both cases, the precise fractal structure of the boundaries of these domains helped a lot to construct a Green function $G$, designed with the help of an appropriate rigid tiling of $\Omega$, 
with invariance properties that allowed to prove the ellipticity of $a$.

In \cite{DP}, $\Omega$ is one of the two components of $\R^2 \sm \Gamma$, where $\Gamma$ is any Reifenberg flat
curve in the plane, with small enough constant (see for instance \cite{DT} for a definition). We still define $G$
first, and check that if we do this well, $G$ solves an elliptic equation (with $L = - \rm{div} \,  a \,  Id \nabla$). 
We even manage to choose $G$ and $a$ so that the elliptic measure is equivalent to any given
measure $\sigma$ on $\d\Omega$, provided that it is doubling with a sufficiently good constant.
This last also seems to be interesting even in the special case when $\Omega$ is a half plane.

The main difficulty is that, with general Reifenberg flat curves, precisely self-similar tilings are now out of the question. We can still draw green curves 
in $\Omega$ that look like what we expect from level sets of a green function $G$, but then the orthogonal curves
(call them red) will not have the correct long-distance behavior (they should roughly tend to stay at comparable distances
from each other) that allows $G$ to solve an elliptic equation (that is, once we compute coefficients $a$ from $G$, to give
an elliptic function $a$).

This is where Theorem \ref{t1} comes into play, as a way to modify slightly our green curves inside rectangles
$R \subset \Omega$, so as to control locally the distances between the various orthogonal red curves, and prevent them
from going crazy at large distances. Again, we expect that similar constructions could be useful in other contexts.

~\ 

\noindent The rest of this article is organized as follows. 
In Section \ref{S2}, we give a summary of the proof of Theorem \ref{t1}, see Subsection \ref{S2S2}. 
The definitions and statements 
in Subsection \ref{S2S1} seem however essential to understand it. Sections \ref{S3}-\ref{S10} are devoted to the proof itself. More guidelines for their contents can be found in the aforementioned summary in Subsection \ref{S2S2}. 
In Section \ref{S11} we solve a particular case of the double control problem presented in Remark~\ref{r1a4}. 
The last Section \ref{S12} answers Question \ref{q1a8} 
and provides an alternative proof for Theorem \ref{t1}.

~\ 

\noindent {\bfseries LLM disclosure:} all the main results 
of this paper are contained in Sections \ref{S1}-\ref{S11}, and, apart from Section \ref{S12}, all mathematical arguments and non-standard definitions in this manuscript are created by the authors. Section \ref{S12} contains an argument produced in its bulk by OpenAI's Chat GPT, but reprocessed and verified independently by the authors, who take full responsibility for the content of the manuscript. This text was also hand-typed and edited without any AI assistance.
 
\section{First steps towards designing the curves and a summary of the construction}
\label{S2}

\subsection{Antisymmetric elementary targets, the classes $\Psi_0(I)$, $\Psi(I)$, and a 
question about finite decompositions} 
\label{S2S1}

As mentioned above, one can interpret our problem as a search for a smooth vector field $n(x, t)$ 
such that the integral red curves $R_x$ of this field, starting at $(x, c)$, arrive at $(f(x), d)$,
while the integral curves $G_t$ of the orthogonal field $e(x,t)$ cross $R$ as in \eqref{1a1}.
This is why we sometimes call the function $f$ a target function, which we try to reach by constructing $n(x, t)$.  

The next lemma says that if we can reach a target $f$ and a target $g$, then, by a simple geometrical construction, 
we can also reach $f \circ g$ and $f^{-1}$ (the inverse mapping). In very broad strokes, to obtain Theorem \ref{t1}, we use this fact, plus simple constructions for specific targets $f$, to get as large a class of reachable targets as possible, in which we approximate our initial target $f$ with a small error. Then, we iterate at smaller scales. 

\begin{lem}  \label{l1a4}
Suppose $f$ and $g$ are two increasing bijections of $I$ for which we can construct the desired net of curves
(the conclusion of Theorem \ref{t1} holds for them), with intervals $J_f$ and $J_g$. Then we can also prove Theorem \ref{t1}
for $f \circ g$, with an interval $J$ of length $|J_f|+|J_g|$, and for $f^{-1}$, with the same interval $J_f$.
\end{lem}
\begin{proof}
We just put $J_f$ on top of $J_g$ vertically, and build in each rectangle $I \times J_f$ and $I \times J_g$ the vector field coming from the assumption; the resulting vector field is still $C^m$, and a red curve that starts from $(x, c_g)$ and follows the vector field $n$ will go through $(g(x), d_g)$ and then will continue to $(f \circ g(x), d)$ as needed, with $d = c_g + (d_g - c_g) + (d_f - c_f) = d_g + (d_f - c_f)$.
See Figure \ref{fig2}. 

Similarly, if we take the solution for $f$ and then apply a symmetry around the horizontal line
$t = (c_f + d_f)/2$, we get a solution for $f^{-1}$.    
\end{proof}
\begin{figure}[!h]
  \centering 
 \includegraphics[width=6cm]{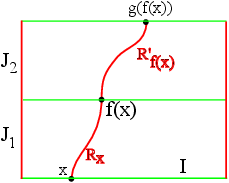} 
  \caption{How to solve for the composition $g \circ f$}
\label{fig2}
\end{figure}

Lemma \ref{l1a4} supplies us with some operations which allow to reach new targets, 
provided that we know already how to reach $f$ and $g$. 
It is time to describe our basic class of easy targets $f$.
Recall that $f$ satisfies \eqref{1a4}, and that we systematically write $f = Id + \psi$, where $Id: \R \to \R$ is the identity and $\psi$ is supported in $I$. 
A simple way to make sure that \eqref{1a1} holds is to make our green curves symmetric
with respect to the vertical line $L = \big\{ x = (a+b)/2  \big\}$. This justifies the following. 

\begin{defi}\label{defi1a9}
For any interval $I = [a,b] \subset \R$, we denote by $\Psi_0(I)$ the set of functions $\psi : \R \to \R$ 
that are at least $C^2$, supported in $I$, and antisymmetric in the sense that
\begin{equation}\label{1a9}
\psi(a+b-x) = - \psi(x) \ \text{ for } a \leq x \leq b.
\end{equation}
That is, the graph of $\psi$ is symmetric with respect to $(x_0, 0)$, with $x_0 = (a+b)/2$. 
Or, $\psi(\frac{a+b}2 +\cdot)$ is odd.
\Bk
We also require that $||\psi'||_{L^\infty} < 1/2$, which among other things is a sure way to force
$f = Id + \psi$ to be an increasing diffeomorphism of $I$.
\end{defi}

The next lemma is not very complex, but it is crucial for what follows. The most important part is very graphical and is presented in Figure \ref{fig3}. However, the estimates on the graphs of functions $r_x$ and $g_t$, which represent the curves, are not pleasant to do. Therefore we suggest, if we may, that the reader looks at the part around Figure \ref{fig3} (the first page or so of the proof), trusts us with the estimates on the derivatives of $r_x$ and $g_t$, and skips the rest of the proof if they find themselves annoyed by the estimates. All this so that they may reach the summary of the whole construction relaid in the beginning of Subsection \ref{S2S2} faster.  

\begin{lem}  \label{l1a5}
Let the intervals $I$ and $J$ be given. 
The conclusion of Theorem \ref{t1} holds when $\psi \in \Psi_0(I')$ for any $I' \subset I$ under the assumption that $N_0(\psi) = ||\psi||_\infty \leq 20^{-1} |J|$, and $N_1(\psi) \leq 1/2$ as above. 
 
Furthermore, the green and red curves are $C^k$ as soon as $\psi \in C^k$, and we have the following estimates for the functions $g_t$ (whose graph is $G_t$) and $r_x$ (whose graph is $G_t$):
\begin{equation} \label{2bb2}
||\nabla_{x,t}^j \, g_t(x)||_\infty + ||\nabla_{x,t}^j \, r_x(x)||_\infty \leq C \Big(1+ \frac{|I|}{|J|}\Big)^{j+1}
\sum_{i \leq j} N_i(\psi) 
\, \text{ for } 1 \leq j \leq k.
\end{equation}
\end{lem} 

When $|J| < |I|$, our estimate is clearly not optimal, and probably the correct power is $k$; see Remark \ref{r2a4}. 
For a more precise estimate on the $r_x$ alone (and all $I$, $J$), see \eqref{2c4}.

\begin{proof}
It is enough to prove this when $I'$ is symmetric with respect to the origin, 
and even $I = I'$ (we will not move anything in $\R \sm I'$).
In addition, in order to simplify our proofs, let us assume that $I= [-1, 1]$ and, for the moment,
any $J = [c,d]$. In this case $\psi$ is just odd (and our formulas will be simpler).

We first construct and study the red curves $R_x$, $x \in I$. 
Because of \eqref{1a2}, we want
$R_x$ to start at $(x, c)$ and end at $(f(x), d) = (x+\psi(x), d)$.
Let us decide that $R_x$ will be the graph (over the $t$-axis) 
of a function of $t \in J$, that is, $F_x: [c,d] \to \R$, chosen so that $F_x(c) = x$ and $F_x(d) = x + \psi(x)$. 
For instance we choose a nice function $\varphi$ such that $\varphi(t) = 0$ near $0$, 
$0 \leq \varphi \leq 1$ everywhere, and $\varphi(t) = 1$ near $1$, and then set 
\begin{equation}\label{1a10}
F_x(t) = x + \psi(x) \varphi(t)  \ \text{ for } t \in J.
\end{equation}
See Figure \ref{fig3}. 
We also extend $F_x$ to $\R \sm J$ by taking $F_x(t) = x$ for $t < c$ and $F_x(t) = x+ \psi(x)$ for $t > d$.
Notice that $F_x(t) = x$ for $x = \pm 1$; we also set $F_x(t) = x$ for $x \in \R \sm I$. 
Notice that then we have \eqref{1a3}.

\begin{figure}[!h]
  \centering 
 \includegraphics[width=9cm]{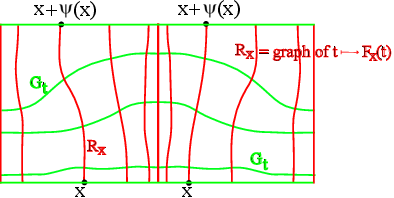} 
  \caption{Antisymmetric data $\psi$}
\label{fig3}
\end{figure}

Obviously $(x,t) \mapsto F_x(t)$ is at least as regular as $\psi$, so we will not lose regularity here.
At this point most of the proof is in the picture, but let us still formally check all the properties from pages 3 and 4, and write down some estimates on derivatives (will be needed later). 

\ms
First, let us verify that the red curves form a foliation of $I \times J$. To do that, consider the mapping $H$ defined by $H(x,t) = (F_x(t),t)$, and check that it is a 
diffeomorphism of $I \times J$. $H$ is a bijection, since for $t \in J$ given,
$x \mapsto F_x(t) = x + \psi(x) \varphi(t)$ is differentiable, with a positive 
derivative, and preserves the endpoints of $I = [-1, 1]$. Let us also record for later that its Jacobian matrix at $(x,t)$ is 
\begin{equation} \label{2c3}
D(x,t) = \begin{bmatrix} 1+\psi'(x) \varphi(t) & \psi(x) \varphi'(t)  \\
      0 &1 
     \end{bmatrix} .
\end{equation}
Our foliation condition for the $R_x$ follows, because for $(y,t) \in I\times J$,
saying that $(y,t) \in R_x$ amounts to saying that $y = F_x(t)$ or that 
$(y,t) = H(x, t) = (F_x(t),t)$, and the unique solution for $x(y)$ exists.

We made sure $R_x$ starts from $(x,0)$ and ends at $(x + \psi(x), 1)$, as needed for 
\eqref{1a2}, and the only thing that we did not check yet is that 
$R_x$ is a $10^{-1}$-Lipschitz graph. 
In the definitions concerning the $R_x$, 
we also required that $R_x$ is the graph of a function $r_x$, and so here $r_x = F_x$, 
and it is Lipschitz with constant $||\varphi'||_\infty ||\psi||_\infty$, and we now deduce the desired Lipschitz bound 
from our requirement that $||\psi||_\infty \leq \frac{1}{20} |J|$ (we took $\frac{1}{20}$ just to be safe). 
The norms for the $r_x$ are also easy to obtain: if we choose $\varphi$ correctly, with 
$|\varphi^{(i)}| \leq C_i |J|^{-i}$, we immediately have \begin{equation} \label{2c4}
|\nabla_t^i \nabla_x^j r_x|
=  |\nabla_t^i \nabla_x^j F_x| \leq C_{i,j} |J|^{-i} N_j(\psi) 
\end{equation}
for $i > 0$ and $0 \leq j \leq k$, while for $i=0$ we get $|\frac{\d r_x}{\d x} - 1| \leq C N_1(\psi)$
and $|\nabla_x^j r_x| \leq C N_j(\psi)$ for $2 \leq j \leq k$.

\medskip
For the green curves we naturally need to work more, 
because they are given in a less explicit manner. 
We first compute tangents and normals. Set 
\begin{equation} \label{2c5}
n_0(x,t) = (\psi(x) \varphi'(t), 1);
\end{equation}
this is a tangent vector to $R_x$ at the point $H(x,t) = (F_x(t),t)$.
Said in other words, 
\begin{equation}\label{2c6}
n(x,t) = n_0(H^{-1}(x,t)).
\end{equation} 
is a tangent vector for the $R$-curve that goes through $(x,t)$.

Let us assume that $\psi$ is of class $C^k$ (for some $k$, which may be smaller than the one in the theorem).
Obviously, $n_0(x,t)$ is also $C^k$, as well as $n(x,t)$ because $H$ is a diffeomorphism of class $C^k$.
Amusingly, we did not lose a derivative here because we only differentiated $\varphi$ to get $n_0$ and $n$.
Call 
\begin{equation}\label{2c7} 
e_0(x,t) = (1, - \psi(x) \varphi'(t))
\ \text { and } e(x,t) = e_0(H^{-1}(x,t))
\end{equation}
two vectors that we choose to be orthogonal to $v(x,t)$ and $n(x,t)$.
Because of \eqref{1a7} we want 
$e(x,t)$ to give the direction of the tangent to the $G_s$ that passes through $(x,t)$.  
Fortunately, even when $k=1$, $e$ is sufficiently continuous/Lipschitz continuous for us to apply the Cauchy-Lipschitz theorem
and find curves $G_t$, $t \in J$, that foliate $R$. In addition, 
since $n(x,t)$ is almost vertical, $e(x,t)$ is almost horizontal, and $G_t$ is automatically a $10^{-1}$-Lipschitz
graph. Call $g_t : \R \to \R$ the function whose graph is $G_t$; we know that $(x,t) \mapsto g_t(x)$
is $C^k$ (the estimate below will say even more). But we also want to estimate its derivatives, which will be the unpleasant part. 

\medskip

We need more notation. It will be convenient to write 
\begin{equation} \label{2c8}
H(x,t) = (h_t(x), t), \ \text{ with } h_t(x) = F_x(t) = x + \psi(x) \varphi(t)
\end{equation}
and $H^{-1}(y,t) = (f_t(y), t)$, where the equation $H \circ H^{-1}(y,t) = (y,t)$ becomes
$h_t(f_t(y)) = y$. In other words, $f_t$ is the inverse of $h_t$ for the composition. 
Let us first fix $t$, remove it from or notation, and estimate the derivatives of $f_t$.
We start from $f(h(x)) = x$, and take successive derivatives; we claim that for a composition in general,
\begin{equation}\label{2b9}
(f \circ h)^{(k)}(x) = \sum_{j = 0}^k f^{(j)}(h(x)) \, B_{k,j} \big(h'(x), \ldots, h^{(k-j+1)}(x)\big)
\end{equation}
for some polynomial $B_{k,j}$ of degree at most $k$ that we don't need to compute (these are called Bell's polynomial).
This is easy to check by induction.
We apply this with $k \geq 1$ when $f \circ h(x) = x$. 
Notice that, for $j = k$, there is only one term in the polynomial $B_{k, j}$ of \eqref{2b9}, and the corresponding 
term in the sum is $f^{(k)}(h(x)) h'(x)^k$. This allows us to 
compute it from $(f \circ h)^{(k)}(x)$
(here equal to either $0$ or $1$) and the previous $f^{(j)}(h(x)) h'(x)^j$. We iterate and get that 
\begin{equation}\label{2b10}
f^{(k)}(h(x)) \,  h'(x)^k = P(h(x), h'(x) \ldots, h^{k}(x))
\end{equation}
for some polynomial $P$. Since $|h'(x) -1| = |\psi'(x)| \varphi(t) \leq 1/2$ 
because we assumed that $N_1(\psi) < 1/2$, we get that 
$|f^{(k)}(h(x))| \leq 2^k |P(h(x), h'(x) \ldots, h^{k}(x))|$. 
Set $\ol N_k(\psi) = \sum_{j=0}^k N_j(\psi)$; if $\ol N_k(\psi)$ is small enough
(this will be our main case), we obtain that
\begin{equation}\label{2b11}
|f^{(k)}(h(x))| \leq C \ol N_k(\psi) \, \text{ for } k \geq 2.
\end{equation}
Finally, for $k=1$,
we already observed that $|h'(x) - 1| \leq |\psi'(x)| \leq 1/2$, so
$f'(h(x)) = 1/h'(x) \sim 1$.

Now return to the more explicit notation with $f(x) = f_t(x) = f(x,t)$ and $h(x) = h_t(x) = h(x,t)$, with $t \in J$. 
We want to take partial derivatives, in $t$ as well, of $H^{-1}(x,t)= (f(x,t), t)$. In particular, \eqref{2b10} writes as 
\begin{equation}\label{2b12}
[\d_x^k f](h(x,t),t) \,  [\d_x h(x,t)]^k = P(h(x,t), \d_x h(x,t) \ldots, \d_x^k h(x,t)),
\end{equation}
and we have similar formulas for $[\d_x^j f](h(x,t),t)$, $j < k$.
When $j < k$ and we differentiate $\d_x^j f(h(x,t),t)$ with respect to $t$, we get 
$$
\frac{\d}{\d t}[\d_x^j f(h(x,t),t)] = [\d_t\d_x^j f](h(x,t),t) \cdot \d_t h(x,t) + [\d_x^{j+1} f](h(x,t),t) \cdot \d_t h(x,t).
$$

On a slightly informal level, the first term on the right-hand side is equivalent to $[\d_t\d_x^j f]$ $(h(x,t),t)$ because $\d_t h(x,t)$ is equivalent to $|J|^{-1}$. It is thus computable in terms 
of $h(x,t),$ $\d_x h(x,t)$ $ \ldots, \d_x^{j+1} h(x,t)$, and the $t$-derivative of the right-hand side of \eqref{2b12}. 
We can iterate this computation, i.e., take more $t$-derivatives and compute the main term in terms
of the other ones (that we already know how to compute) as long as we do not reach derivatives of order $k+1$. 
This way we find formulas for all the partial derivatives $[\d_t^\ell \d_x^j f(h(x,t),t)]$, as long as $\ell + j \leq k$.
We get expressions in terms of $h$ and its derivatives, which are polynomial, but 
only after we divide by $h_t'(x)= \d_x h(x,t)$ to some power. So we claim that
\begin{equation} \label{2b13}
|\d_t^\ell \d_x^j f| \leq C_{i, j} |J|^{-l} \ol N_k(\psi)  \ \text{ when $j \geq 1$ and $\ell \geq 0$ are such that } \ell + j \leq k.
\end{equation}
We need to be careful, because when we write down $\d_t^\ell \d_x^j f$ in terms 
of $h$ and its derivatives, some terms may be large. For instance, we have seen below \eqref{2b11} that 
$|f_t'(x)-1| \leq C N_1(\psi)$ (and hence $|f_t(x)-x| \leq C N_1(\psi)$ as well), and $1$ and $x$ are too large
for \eqref{2b13}. Yet all the terms that contain a $t$-derivative of $h$, or at least two $x$-derivatives
of $h$ (or more generally that contain $\psi$ or $\varphi$ in their expression) comply with \eqref{2b13}.
The other terms must be present also when $\varphi = 0$, where $H^{-1} = H = I$, and these terms vanish when
$j \geq 2$ and $\ell > 0$. So \eqref{2b13} holds, and since $f$ is the main part of $H^{-1}$, we get the same estimates for
$H^{-1}$. 

This implies that the vector field defined by \eqref{2c7} is $C^k$ (as we said earlier), and 
\begin{equation} \label{2b14}
|\nabla^l_t\nabla_x^j e_0| \leq C |J|^{-l} \ol N_j(\psi) \text{ for } 1 \leq j + l \leq k.
\end{equation}
Let us return to $G_t$, the graph of $g_t$. We already said that $\xi_t(x) = (x,g_t(x))$ is the 
solution of the flow problem $\frac{\d \xi}{\d x} = e_0(\xi(x))$, with the initial value condition
$\xi_t(0) = t \in \R$. The strong form of the Cauchy-Lipschitz theorem says that 
$(x,t) \mapsto \xi_t(x)$ is of class $C^k$, with estimates that depend only on \eqref{2b14}.
On a large set, the estimates would get exponentially bad, but here the lengths of $I$ and $J$ are $\leq 2$, 
and we get the desired estimate 
\begin{equation} \label{2b15}
|\nabla^l_t\nabla_{x}^j \, g_t(x)| \leq C |J|^{-l} \ol N_j(\psi) \text{ for } 1 \leq j \leq k.
\end{equation}

\ms
We finally observe that if $|I|$ is larger, then all our estimates above transform into something which contributes another factor $|I|^{j + 1}$ in \eqref{2bb2}. One can do this by considering $\ol \psi(x) = \lambda \psi(\lambda x)$ on $I$ instead of $\psi$ on $\lambda I$.

\ms
To complete the proof of Lemma~\ref{l1a5}, we still need to check that the curves $G_t$
satisfy the horizontal connection property \eqref{1a1}. This will is easy, because we constructed
each green curve $G_t$ to be symmetric with respect to the vertical axis. That is, 
$x \mapsto F_x(t)$ is odd (by \eqref{1a10}), then $H(-x, t) = (F_{-x}(t), t)$ is the symmetric of 
$H(x,t)$ with respect to the vertical axis, then $H^{-1}$ has the same property. In addition,
$n_0(-x,t)$ is the symmetric of $n_0(x,t)$ by \eqref{2c5}, so $n_0(-x,t)$ is the symmetric of $n_0(x,t)$ 
by \eqref{2c6}, and then the integral curve is symmetric.
\end{proof}

\begin{rem} \label{r2a4}
Our estimate when $J$ is smaller than $I$ is most probably suboptimal. For one thing,
it is possible that when we check the argument, we do not differentiate $\varphi$ more than $k$ times,
and also the homogeneity of norms could be improved.
The following remarks should help the reader get some good estimate if they so desire. 
First, reasonable norms for functions on an interval $I$ of length $\ell$ are the
\begin{equation} \label{}
\ol N_j(\psi) = \sum_{0 \leq i \leq j} \ell^{1-i} N_i(\psi),
\end{equation}
because their invariance allows an easy reduction to an interval of length $1$, and a similar remark applies to function in $J$.
Then, if we desire to modify $J$ alone and leave $I$, we may first construct the $R_x$
as before, with the estimates \eqref{2c4}, and when it comes to constructing and estimating the orthogonal $G_t$, 
observe that dilating $J$ amounts to changing the scalar product in $\R^2$, and eventually multiplying
$\psi(x) \varphi'(t)$ by a constant in the definition \eqref{2c7} of $e_0$. We leave the details.
\end{rem}

\medskip

We continue to expand our class of admissible targets and observe that we 
can treat all at once collections of functions $\psi$ as above, on the condition that they have disjoint supports.

\begin{defi}
For a possibly large interval $I$, we say that a function $\psi : \R \to \R$ belongs to the class $\Psi(I)$ when there exists a finite collection of disjoint intervals
$I_j \subset I$ for which $\psi$ is representable as 
\begin{equation}\label{1a13}
\psi = \sum_j  \psi_j \1_{I_j}, \quad \psi_j \in \Psi_0(I_j) \quad \text{ for all} \; j.
\end{equation}
In other words, 
\begin{equation}\label{1a12}
\psi(x) = \psi_j(x)  \ \text{ on } I_j
\ \text{ and } \psi(x) = 0  \ \text{ on } \R \sm \cup_j I_j.
\end{equation}
\end{defi}

\begin{lem}  \label{l1a7}
The conclusion of  Theorem \ref{t1} also holds when $\psi \in \Psi(I)$,  
with any $J$, and  with the same estimate on the $C^k$ norms as in Lemma \ref{l1a5}. 
\end{lem}
\begin{proof}
We can think that this time we have disjoint rectangles with the same length of the vertical sides, and by Lemma \ref{l1a5}
we know how to build red and green curves on each of them. Then, in the rectangle $I \times J$, we complete the green curves by line segments, draw the vertical missing red curves, and get a solution for $\psi$.
See Figure \ref{fig4}. 

In general, it is trivial to compose mappings $f_j$ for which the $f_j - I$ have disjoint supports, 
but of course we have to keep a control on the $C^k$ norm in the constructed curves.
Here, what really counts is the maximal length of a support, so we may even have a longer $I$ than in the 
previous lemma, provided that the supports are under control. In this paper we will always use decompositions
with a finite number of pieces of the same length.
\end{proof}

\begin{figure}[!h]
  \centering 
 \includegraphics[width=11cm]{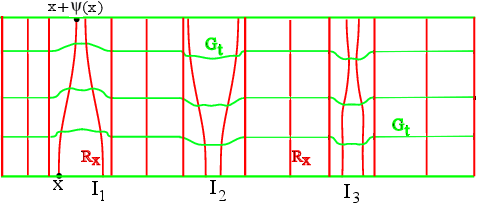} 
  \caption{Sums of antisymmetric data $\psi$ with disjoint supports}
\label{fig4}
\end{figure}

\begin{qn}  \label{q1a8}
If we want a simple proof of Theorem \ref{t1}, it is natural to ask the following question. Given $f$ as in the 
theorem (with $\varepsilon$ in \eqref{1a4} small enough), can we find a finite collection of mappings $\psi_j \in \Psi(I)$, $1 \leq i \leq N$,
so that with $f_j(x) = x + \psi_j(x)$, we have 
\begin{equation} \label{1a14}
f = f_N \circ \ldots \circ f_2 \circ f_1,
\end{equation}
if possible with a bound on $N$, a lower bound on the size of all the intervals of the decompositions \eqref{1a13},
and an estimate like $||\psi_j||_{m} 
\leq \delta(\varepsilon)$ where $\delta(\varepsilon)$ is an acceptable power of $\varepsilon$? 
The smoothness requirements on $\psi_j$ can be reasonably relaxed, as the reader could have guessed, given our main statement and the remarks that followed.  
\end{qn}

If the answer is yes, with all these estimates, we can prove Theorem \ref{t1} 
(maybe with less precise estimates, if $m$ is small and $\delta(\varepsilon)$ large) 
directly with Lemma \ref{l1a7} and a finite pile of rectangles.

\medskip

At the time we proved Theorem \ref{t1} and started to 
prepared this text for publication, we did not know the answer to this question. 
The existence of such a decomposition seemed to us like a lot to ask, and we had no clue  
about how to proceed. But there was some hope because there are lots of decompositions to try.

Yet, given the spectacular advancement of the LLM technologies, we judged it would be best to ask Question \ref{q1a8}
to an AI ourselves. First, it was our colleague Raul Jimenez who kindly transferred our information request to his version of ChatGPT Pro, and then the second author repeated the same experiment with the Plus version of ChatGPT 5.6 Sol, 
with more or less the same outcome. Namely, the answer is yes (modulo preparing our initial target with
the curtain trick of Section \ref{S3}), and it is even not hard to prove with earlier results 
of T. Tsuboi on the structure of the diffeomorphisms of the circle \cite{Ts}.

We give an account of the method in our last Section \ref{S12}, which we wrote by hand, the old way, but using 
ideas given by the AI.

\smallskip
We also decided to keep our old proof, where we replace the exact decomposition required in Question \ref{q1a8}
by a longer and more technical approximation argument. We did not do this entirely out of contempt; the proof by approximation 
may have some advantages that we list now.

First, the result of Tsuboi is not so easy to prove, and probably the total amount of work is still smaller with our 
proof. Yet the structure of the AI proof is simpler, as one only stacks a finite number of intervals $J$; 
see Section \ref{S12}.

The best exponents $m = m(k)$ (how many derivatives we get for the solution) and $1/2$ 
(giving the size of the solution) seem to be different. 
Apparently, adding \cite{MMY} to \cite{Ts} gives an even better $m$, and shorter decompositions. 
For the exponent $1/2$ , we seem to do better because we only use commutators of order $2$, 
and which are probably more explicit because of the construction of Section \ref{S5}.

But the main reason is that we believe our construction to be more flexible, and, in particular, it helps to prove a suitable multidimensional analogue of Theorem \ref{t1}, when, to our knowledge, no analogue of the finite decomposition \eqref{1a14} exists. This is the content of some work in progress. 

\subsection{A summary of the proof of Theorem \ref{t1} and the main approximation result}
\label{S2S2}

Let us finally give a brief overview of how our proof of Theorem~\ref{t1}, based on approximation and wavelets, is going to work. 

\smallskip
Recall that we wish to get a target $f = Id + \psi$ with $\psi$ ``small'', which we can try to reach by compositions, as stated in Lemma \ref{l1a4}. 
Typically, we shall proceed backwards, i.e., try to isolate an interesting part of $\psi$, find a constructible
mapping $\psi^\ast$ (i.e., one for which we can solve the problem of Theorem~\ref{t1}) that incorporates the interesting
part, and compose with the inverse of $\psi^\ast$ to reduce to a simpler target. We will do this many times. But this is equivalent to representing $f$ as a composition of mappings.

As announced, we will not be able to decompose $f$ exactly in finite time, so after finding some number of mappings for the representation of $f$, we will decide that our new target is ``small enough'', consider that we are done with Step 1, and restart with this last target at a smaller operating scale. We need to keep track of the size of some norms of the target, and the effect of the change of scales on those norms. Then we will iterate, prove some convergence, and hopefully conclude.

Let us explain what happens during one step of our iteration algorithm (e.g., during Step~1). To get a representation of $f = Id + \psi$ with $\psi$ ``small'', it makes sense to try to compose mappings $Id + s_i$ with some ``small'' compactly supported functions $s_i$. Although we know that composition is a disagreeable operation, e.g., not commutative, we will often be able to say that the composition is close to $Id + \sum_i s_i$, with second order estimates on the error. We give these basic error estimates in Section \ref{S4}, to be used a few times later. We would like to stress that the observation ``the composition of mappings $Id + s_i$ with small $s_i$ behaves like $Id + \sum_i s_i$'' is not complex, but it is crucial for the idea of this construction. Indeed, once we start thinking of sums, it becomes very natural to try to decompose $\psi$ into a sum of functions we already know how to target, and it should not seem surprising that we aim to do a wavelet decomposition.

So, at the center of Step 1 is the attempt to decompose the current target $\psi$ 
(we will indiscriminately refer to $f$ or $\psi = f - Id$ as targets) into wavelets,
where the wavelets $\psi_j$ will be chosen to be antisymmetric, so that the $Id + \psi_j$ are constructible by 
Lemma \ref{l1a5} or \ref{l1a7}. If we can do this, we will be able to ``subtract'' a finite number of these wavelets from $\psi$, use the fact that composition is almost the same as sum, and have a new target which is smaller, because it only contains errors of order $2$ and a remainder in a wavelet expansion. In Section~\ref{S6} we select a wavelet system that is composed of constructible mappings (by antisymmetry), and prove some estimates on the convergence of partial sums. 

Unfortunately, our initial target (or later, the first target at each stage) cannot always be decomposed
into odd wavelets. This is because our odd wavelets associated with scales $2^{-k}$, $k \geq 0$, 
do not span the whole $L^2$ locally, and typically leave an initial space $V_0$ of finite dimension.
For example, all our wavelets have integral zero, while this may not be true of our initial target
$\psi$. This is connected to Remarks \ref{r1a2} and \ref{r1a3}, because the wavelets may span the whole set of first order targets obtained by linearising our problem, but not the whole space of targets. 

Therefore, there is a certain number of basic targets, such as the small elements of the initial vector space $V_0$
in our wavelet construction, that we will need to construct by hand. We will not obtain the basis
of $V_0$ directly, but what we do will be enough to approximate its elements by constructible mappings. These elements are combinations of simple explicit {\bfseries commutators}, which we define and estimate 
in Section \ref{S5}, and the approximation of elements of $V_0$ will be done in Section \ref{S6}.

In short, at our Step 1, we take our initial target $\psi$, maybe make its support smaller using the curtain trick of Section \ref{S3}, remove a few commutators that approximate the $V_0$ part of $\psi$ from the representation,
then remove a big part of the wavelet expansion of the rest, and, at the end, obtain a new much smaller (in sup norm) target (because confusing compositions with sums only produces errors of order $2$, the $V_0$
part was well approximated by commutators, and the rest by wavelets). We make sure (in Section \ref{S8})
that with all these modifications, we do not increase $N_k(\psi)$ by much (much less than what we gain on the sup norm), and we ensure we can change scales and start Step 2 on better grounds than we had for Step 1.

Then, as promised, we change scales, reproduce Step $1$ at the new scale, iterate, check that our targets tend to the identity fast enough, collect estimates on the constructed curves, and conclude. This is the place where we find out that our $m$ is smaller than expected. We do all this in Section \ref{S9}.

\ms
We end this introduction with a statement of our main approximation result, which corresponds to one iteration in the algorithm described above. We first leave out some details to keep it almost readable, and give some comments.

\begin{thm}  \label{t1a9} 
Let $N \geq 1$ be large enough, and $\varepsilon > 0$ small enough (depending also on $k$ and $N$).
Take $f = Id + \psi$ as in Theorem \ref{t1}, with  $I = [0,1]$, and in particular assume that $\psi$ satisfies \eqref{1a4}. 
Then we can find a constant $C_0 \geq 1$, depending on $k$, and a finite collection of $C_0 N$ mappings $\psi_j \in \Psi(I)$, $1 \leq i \leq C_0 N$,
so that if we set $f_j(x) = x + \psi_j(x)$, and, for $1 \leq M \leq C_0 N$, 
\begin{equation} \label{1a15}
F_M = f_M \circ \ldots \circ f_2 \circ f_1,
\end{equation}
we have
\begin{equation} \label{1a16} 
N_k(F_M) := ||F_M^{(k)}||_{L^\infty} \leq C(1 + M)^2 N_k(\psi) + CN_0(\psi)
\ \text{ for } 1 \leq M \leq C_0 N
\end{equation}
and
\begin{eqnarray} \label{1a17}
||F_{C_0N} - f ||_{L^\infty} &\leq& C 2^{-N k} (N_k(\psi)+N_0(\psi)) + C (1+N)^2 (N_k(\psi)+N_0(\psi))^2  \nn\\
&\,& \hskip30pt+ C(1+N)^3 (N_k(\psi)+N_0(\psi))^3,
\end{eqnarray}
and all the intervals that show up in the decompositions of the $\psi_j$ have a length at least $2^{-N}$.

The conclusion also holds for intervals $I$ of length at least $1$, with the same constants, but under the additional condition that $||\psi||_\infty \leq \varepsilon$.
\end{thm}

We added $CN_0(\psi)$ in \eqref{1a16} and \eqref{1a17} just to account for the case when 
$I$ is long or $\psi$ is not supported in an interval of length $1$; otherwise, $N_0(\psi)$ 
is controlled by $N_k(\psi)$ anyway. 
Also, we did not try to make \eqref{1a17}
look more beautiful: we just copied what we got. Finally, we used exactly $C_0 N$ mappings to simplify the 
notation, but if we need fewer, we can compose a few times with $f_M = I$ to comply with the statement.

Now let us explain what us going on in the statement of Theorem \ref{t1a9}. 
The parameter $N$ will be the number of generations that we use in a wavelet construction one step at a time, and we will need less than $C_0$ sub-collections of wavelets for each generation (so that the supports are disjoint each time), plus a few at the beginning (we do not apply wavelets immediately, as announced); the estimate \eqref{1a16} states that the $N_k$ norm of the objects we construct only increases slowly in $N$.

On the other hand, the point of \eqref{1a17} is that the $L^\infty$ approximation numbers go much faster to $0$. 
The reader should concentrate on the first term $2^{-N k} (N_k(\psi)+N_0(\psi))$, which decays fast, and not worry about the 
other terms, which look horrible but are of second and third order (in $\varepsilon$, or the $N_k$), and will be controlled easily
(by making $\varepsilon$ small at the end, before we iterate).

Note that one should not be alarmed by the dependence of $\varepsilon$ on $N$ mentioned in the statement of the theorem. The concern could arise from the fact that we iterate the result of Theorem \ref{t1a9} to obtain the conclusion of Theorem \ref{t1}, and it would be problematic to choose different $N$ at each iteration. Luckily, the number $N$ of generations that we use in a wavelet construction will be the same each time, as we will see in Section \ref{S9}, so there is no problem.

As mentioned already, once we prove Theorem \ref{t1a9}, we will change scales so that the normalized $N_k(f_{C_0 N})$ remains small, and we will iterate the construction. This will allow us to approximate our initial target $f$ by a sequence of targets $f_j$ that are compositions of constructible mappings, and that will converge to $f$ 
in some lesser $C^m$ norm.
Then we will use these mappings to construct a tower of curves and prove Theorem \ref{t1}.

\medskip

There is a detail in the algorithm that we have thus far almost kept under the carpet. 
When we iterate, we no longer know if the new target $\psi$ for the next generation 
(computed from $f$ and $F_N$) is compactly supported by intervals of the next generation (which will be much smaller). So, even for Theorem \ref{t1}, we will need to consider the case of functions that are not necessarily 
supported in $I$, but we compensate by assuming that
\begin{equation} \label{1a18}
N_0(\psi) = ||\psi||_{L^\infty} \leq \varepsilon.
\end{equation}
This will not make things harder, thanks to a manipulation that we call the curtain trick, which reduces matters
to the case when $\psi$ vanishes near both endpoints of $I$. We describe it just below in Section \ref{S3}.

\section{The curtain trick and our first change of targets}
\label{S3}

Let $f = Id + \psi$ be as in Theorem \ref{t1}. Even though $\psi$ is supported in $I = [0,1]$, 
we don't like too much working too close to the endpoints of $I$. 
So we first construct an easy target $f_c = Id + \psi_c$ with the same values near the endpoints of $I$ as $f$, 
and composing with $f_c^{-1}$ will allow us to assume that $\psi$ vanishes near the boundary of $I$. 
Later on, we will work with functions that are not even supported on $I$, but will be small in sup norm,
and we will use the same trick to reduce to compactly supported functions.

Let us do the trick near $0$, and discuss the general case later. Let $\varphi$ be a symmetric smooth cut-off
function, with in particular $\varphi(x) = 1$ on $[-\frac{1}{20}, \frac{1}{20}]$, and $\varphi = 0$ on $\R \sm [-\frac{1}{10}, \frac{1}{10}]$.
Then define $\psi_c$ by 
\begin{equation} \label{2e1}
\psi_c(x) = \varphi(x) \psi(x) - \varphi\left(\frac{1}{5} - x\right) \psi\left(\frac{1}{5} - x\right).
\end{equation}
Thus we make $\psi_c$ antisymmetric with respect to $\frac{1}{10}$. It is supported
on $[0, \frac{1}{5}]$ if $\psi$ was supported on $I$ (as in the theorem), and only on
$[-\frac{1}{10}, \frac{3}{10}]$ if $\psi$ had support on the negative half-ray of $\R$ (and the support on the positive half-ray beyond $\frac{3}{10}$ does not count at all). 

Observe that, since $\psi_c$ is antisymmetric on its support, $f_c = I + \psi_c$ is one of our easy antisymmetric targets. 

Let us now check its derivatives to see that composing with it
does not destroy our estimates. First of all,
\begin{equation} \label{2e2}
N_0(\psi_c) \leq N_0(\psi) \leq \varepsilon
\end{equation}
because $\varphi(x)$ and $\varphi(\frac{1}{5} - x)$ have disjoint supports.
Next, for the homogeneous $k$-norm, since $\varphi$ and its derivatives are bounded, we 
claim that 
\begin{equation} \label{2e3}
N_k(\psi_c) \leq C \sum_{j \leq k} N_j(\psi) \leq C N_0(\psi) + C N_k(\psi) \leq C \left(\varepsilon + N_k(\psi)\right);
\end{equation}
where we only need to say why $N_j(\psi) \leq C N_0(\psi) + C N_k(\psi)$ for $1 \leq j \leq k$.
Let $P_k(x) = \sum_{0 \leq j \leq k} a_j X^j$ be the Taylor polynomial of $\psi$ at $0$; thus
$a_j = \frac{1}{j!} \psi^{(j)}(0)$. Then on $[0,1]$, $N_0(\psi - P_k) \leq C N_k(\psi)$
by multiple integrations or directly the Taylor formula with integral remainder. So 
$N_0(P_k) \leq N_0(\psi) + C N_k(\psi)$, and then by equivalence of norms on the vector space
of polynomials of degree $\leq k$, $a_j \leq C N_0(P_k) \leq C N_0(\psi) + C N_k(\psi)$ for all $j$.
Finally $|\psi^{(j)}| \leq C a_j + C N_k(\psi)$ on $[0,1]$, by the the Taylor formula again, which proves \eqref{2e3}.

Now recall that we want to approximate $f_\psi = Id + \psi$ by a composition of constructible mappings, 
and our first one will be $f_c$. That is, we write 
\begin{equation} \label{2e4}
f_\psi = f_{\psi_1} \circ f_c \, ,
\end{equation}
which is the same as $f_{\psi_1} = f_\psi \circ f_c^{-1}$, and our new target will now be
\begin{equation} \label{2e5}
\psi_1 = f_{\psi_1} - Id = f_\psi \circ f_c^{-1} - Id. 
\end{equation}
The clear advantage of $\psi_1$ is that, since $\psi_c = \psi$ on $[-\frac{1}{20}, \frac{1}{20}]$ by construction, 
we also have $f_c = f_\psi$ on $[-\frac{1}{20}, \frac{1}{20}]$, hence $f_\psi \circ f_c^{-1} = Id$ on $[-\frac{1}{30}, \frac{1}{30}]$: naturally we lose a bit 
in the size of the interval, but not too much, since $N_1(\psi_c)$ is small, so $f_c$ is almost $1$-Lipschitz. Thus, we get that
\begin{equation} \label{2e6}
\psi_1 = 0 \text{ on } \left[-\frac{1}{30}, \frac{1}{30}\right].
\end{equation}

We can do the same thing near the other endpoint $1$ of $I$, even at the same time because our two
transformations are independent. Then we also get that 
\begin{equation} \label{2e7}
\psi_1 = 0 \text{ on } \left[1 - \frac{1}{30}, 1 + \frac{1}{30}\right].
\end{equation}
Altogether, our new target $\psi_1$ is such that
\begin{equation} \label{2e8}
\psi_1 \text{ is supported on } \left[\frac{1}{30}, 1 - \frac{1}{30}\right],
\end{equation}
as long as $\psi$ was supported on $I$. 

We will call this the curtain trick (because we hide the sides of the window!). 
To apply it successfully, we also need to check that 
$\psi_1$ is nearly as good as $\psi$, i.e., that 
\begin{equation} \label{2e9}
N_k(\psi_1) \leq C N_k(\psi) + C N_0(\psi).
\end{equation}
This follows from \eqref{2e2} and \eqref{2e3} with the aid of computations from the next section, 
which we prefer to do only once, for a more general case. See below \eqref{3e13}.

\section{Composition, inverses, and nonlinearity}
\label{S4}

We shall often need to revise targets (as above in \eqref{2e4} and \eqref{2e5}), 
and then estimate the norms of the revised $\psi$-mappings (as in \eqref{2e9}). 
In this section we do the basic norm computations for compositions and inverses. 
This is also the place where we explain, among other things, 
why a composition of maps $f_\varphi = Id + \varphi$ and $f_\psi = Id + \psi$ behaves like $Id + \varphi + \psi$, as announced.

\medskip

Recall that we only consider mappings $\psi$, $\varphi$ such that 
$N_1(\varphi), N_1(\psi) < 1/2$, so $f_\varphi$ and $f_\psi$ are increasing $2$-bilipschitz  bijections, 
although we are interested in the case of small norms of $\psi$ and $\varphi$. We also assume that they are supported in intervals of length $1$. Then the main norms to study are $N_k(\varphi)$ and $N_k(\psi)$, 
because the intermediate $N_j$ are controlled by \eqref{1a6}.

We start with estimates for compositions. Although we could bypass it, it sounds clearer to define a nonlinear (and noncommutative) law $\ast$ as follows.
\begin{defi}
    For $\psi, \varphi$ as above, set 
\begin{equation} \label{3e1}
\psi \ast \varphi = f_\psi \circ f_\varphi - Id = (Id + \psi) \circ (Id + \varphi) - Id. 
\end{equation}
More explicitly,
\begin{equation} \label{3e2}
\psi \ast \varphi(x) = (Id + \psi)(x + \varphi(x)) - x = \varphi(x) + \psi(x+\varphi(x)).
\end{equation}
\end{defi}
\begin{lem}
    We claim that for $N_k(\varphi) + N_k(\psi) < 1$, say, 
\begin{equation} \label{3e3}
N_k(\psi \ast \varphi) \leq N_k(\varphi) + N_k(\psi) + C N_k(\psi) N_k(\varphi). 
\end{equation}
\end{lem}
\begin{proof}
Set $f = f_\varphi$ to save notation. Clearly, we only need to estimate the derivatives of $\psi(x+\varphi(x)) = \psi \circ f(x)$.
When we differentiate a total of $k$ times, there is one special term where all the derivatives are taken on $\psi$, which is 
\begin{equation} \label{3e4}
A(x) = (f'(x))^k \psi^{(k)}(f(x)) = (1+ \varphi'(x))^k \psi^{(k)}(f(x)).
\end{equation}
Recall that our norms are small, so
\begin{equation} \label{3e5}
|A(x) - \psi^{(k)}(f(x))| \leq C |\psi^{(k)}(f(x))| |\varphi'(x)| \leq C N_k(\psi) N_1(\varphi) \leq C N_k(\psi) N_k(\varphi),
\end{equation}
which fits \eqref{3e3}. For the other terms, $f'$ is differentiated at least once, 
so the large piece of the derivative of $f$ disappears, and there is always at least one 
$N_j(\varphi)$ and one $N_\ell(\psi)$ in each term. 
This gives \eqref{3e3}, possibly with a large $C$ (depending on $k$).
\end{proof}

The estimate $\eqref{3e3}$ is what we mean when we say that the composition of $f_\psi$ and $f_\varphi$ behaves like $Id + \psi + \varphi$: if $\psi$ and $\varphi$ are small in the $N_k$ norm, the error is 
of order $2$, much smaller  than $N_k(\psi)$ or $N_k(\varphi)$.

\ms
We will also need an estimate for the $\psi$-function associated with the inverse bijection 
$f_\psi^{-1}$ of $f_\psi = I + \psi$, which we call 
\begin{equation} \label{3e6}
\wt\psi = f_\psi^{-1} - I.  
\end{equation}
\begin{lem}
For the symbol $\wt\psi$ of the inverse mapping, one has 
\begin{equation} \label{3e10}
N_k(\wt\psi) \leq N_k(\psi) + C N_k(\psi)^2,
\end{equation}
i.e., the inverse is not much worse than the mapping itself.
\end{lem}
\begin{proof}
We claim that 
\begin{equation} \label{3e7}
||\wt\psi^{(k)}+ \psi^{(k)}\circ f_\psi^{-1}||_{L^{\infty}} \leq C N_k(\psi)^2.
\end{equation}
To show this, we use the computation above with $\varphi = \wt \psi$, and we use the fact that $\psi \ast \varphi = 0$: indeed, $f_\psi \circ f_\varphi  = f_\psi \circ f_\psi^{-1}= I$.
The $k$-th derivative that we compute is $0$, and, at the same time, it is the sum of $\varphi^{(k)}$, 
the main term in $A$ which is  $\psi^{(k)} \circ f_\varphi$, and an error term that is less than $C N_k(\varphi) N_k(\psi)$
by \eqref{3e5}. 
That is, 
\begin{equation} \label{3e8}
|\varphi^{(k)} + \psi^{(k)} \circ f_\varphi| \leq C N_k(\varphi) N_k(\psi).
\end{equation}
Then, by the triangle inequality and because composing with $f_\varphi$ does not change sup norms,
\begin{equation} \label{3e9}
N_k(\varphi) \leq N_k(\psi) + C N_k(\varphi) N_k(\psi).
\end{equation}
If $N_k(\psi)$ is small enough (smaller than $\frac{1}{2C}$), this yields $N_k(\varphi) \leq 2N_k(\psi)$, and then \eqref{3e7} follows from \eqref{3e8}.
When $N_k(\psi)$ is not small, we will not need the estimate, but it is true anyway
because we assume that $N_1(\psi)\leq1/2$, so $f_\psi$ to be $2$-bilipschitz, 
and we can compute $\wt\psi^{(k)}$ brutally. 
The estimate $\eqref{3e10}$ follows, again because composing does not change the sup norm. 
 \end{proof}

\ms
We were interested in \eqref{3e10} mainly because, in the sequel, it will often happen that we have a target $\psi$ and its approximation $\psi_a$, and we want to write $f_\psi$ as a composition
\begin{equation} \label{3e11}
f_\psi = f_{\psi_1} \circ f_{\psi_a}, \text{ or equivalently } f_{\psi_1} = f_{\psi_a} \circ f_\psi^{-1},
\end{equation}
and, among other things, we will need to check that $N_k(\psi_1)$ is not much larger than $N(\psi)$.
With the notation \eqref{3e1} above, \eqref{3e11} says that
\begin{equation} \label{3e12}
\psi_1 = \psi_a \ast \wt\psi,
\end{equation}
and then, by \eqref{3e3} and \eqref{3e10},
$$
N_k(\psi_1) \leq N_k(\psi_a) + N_k(\wt\psi) + C N_k(\psi_a) N_k(\wt\psi)$$
\begin{equation} \label{3e13}
\leq N_k(\psi_a) + N_k(\psi) + C N_k(\psi)^2 + C N_k(\psi_a) N_k(\psi)  + CN_k(\psi_a)N_k(\psi)^2.
\end{equation}
The main point of this estimate is that our errors are small and additive (we never multiply a norm by $2$). 

In the case of the curtain target $\psi_1$ obtained from $\psi_a = \psi_c$
(compare \eqref{2e4} with \eqref{3e11}), we get that 
$N_k(\psi_1) \leq C N_k(\psi) + C N_0(\psi)$ by \eqref{2e2} and \eqref{2e3}
as announced in \eqref{2e9}.

\ms

To wrap up this section and comment some more on the estimates which lead to the intuition ``$\psi \ast \varphi$ behaves like $\varphi + \psi$'', let us remark that, in some cases, we will not be very happy with \eqref{3e13} alone.
Instead, we will need to prove that, for instance, $N_0(\psi_1)$ is much smaller than $N_0(\psi)$ or $N_k(\psi)$, typically because $\psi_a$ was well chosen. This will use a more precise nonlinear estimate for $\psi \ast \varphi - \varphi - \psi$, which we will need to be careful about, since it 
will involve an additional derivative. Let us show what these estimates will look like.

Given $\varphi$ and $\psi$ as above, we are interested in the size of the nonlinearity
$$
\delta = \delta(\psi, \varphi) = \psi \ast \varphi - \varphi - \psi
= \psi \circ f_\varphi - \psi
$$
by \eqref{3e2} (with $f_\varphi = I + \varphi$).
We want to estimate $N_j(\delta)$ (for $0 \leq j \leq k$),  
and we claim that the main term is 
$$
A_\delta(x) = \psi^{(j)} \circ f (x) - \psi^{(j)}(x) = \psi^{(j)}(x+\varphi(x)) - \psi^{(j)}(x),   
$$
which we intend to estimate brutally by
$$
|A_\delta(x)| \leq |\varphi(x)| \, ||\psi^{(j+1)}||_{L^\infty}. 
$$
As we said earlier, this is dangerous, especially when $j=k$, but we'll apply this only when $\psi$ has a controlled
$(k+1)$ regularity (it will be a wavelet). Estimating the other terms of $N_k(\delta)$ as above, one could get
$$
\delta(\psi, \varphi) \leq  |\varphi(x)| \, ||\psi^{(k+1)}||_{L^\infty} 
+ C N_k(\varphi) N_k(\psi),
$$
but we'll do this computation later, slightly differently.

\section{The beauty of commutators}
\label{S5}

Most of our ``approximations'' ($\psi_a$, as in the previous section) of $\psi$ will be antisymmetric wavelets, because they are constructible, but alone they do not span the whole desired space, as was mentioned already (e.g., they all have integral zero). Therefore, a related question is: can we construct (using compositions) non-trivial mappings $f$ such that $f(x) \geq x$ for all $x$ (so that globally all the red curves tend to turn right)? 

This is not so easy to do at first sight, so we were puzzled for some time about whether or not the answer to the question is positive. Vaguely, the idea why is that we almost have a (noncommutative) group here, and there is no reason why this group should not be dense. Plus (after the fact), commutators are often interesting elements in noncommutative groups. 

In this section we estimate commutators (for $\ast$) of some elementary constructible mappings, 
which we shall be able to use to project on a first layer of a basis of wavelets (the ones that don't
have a vanishing integral). Not too surprisingly, the effect of the commutator will be of order $2$
compared with the functions we put in; this is the main reason why some of our estimates are no as good as we could hope, and the main cost for these estimates will be in the size of the norms of the graphs that we construct from these commutators. 

\medskip
To describe our examples, the simplest will be to work on $I = [0,5]$, and later translate and dilate the support. Our two functions (usually $\varphi$ and $\psi$) will be of the form $sa$ and $tb$, where $s$, $t \in \R$ are small parameters; we shall even choose
$s=t$ to simplify computations later. 

Given a function $a$ supported on $[0,4]$, we set
\begin{equation}\label{4b1}
b(x) = a(x-1) \text{ for } x \in \R,
\end{equation}
so $b$ is supported on $[1,5]$. Then we define $f = f_{sa}$ and $g = f_{tb}$ by 
\begin{equation}\label{4b2}
f(x) = x + s a(x)  \ \text{ and } g(x) = x + tb(x),
\end{equation}
which will be constructible because $a$ will be antisymmetric with the midpoint $x=2$
(and $b$ with $x=3$).
We are interested in the commutator 
\begin{equation}\label{4b3}
R = g^{-1} \circ f^{-1} \circ g \circ f ,
\end{equation}
and in particular in proving the fact that $R(x) \geq x$ on $\R$. Since $f$ and $g$ are 
increasing (if $s$ and $t$ are small enough), $R(x) \geq x$ is the same as
$f \circ g \circ R(x) \geq f \circ g(x)$, i.e., as
\begin{equation}\label{4b4}
g \circ f(x) \geq f\circ g(x).
\end{equation}
So we will compute the difference 
\begin{equation}\label{4b5}
\Delta = g \circ f - f \circ g,
\end{equation}
show that it is positive in our examples, and even compare its average size with the product $st$
(equal to $t^2$ in our case), because we want uniform nondegeneracy estimates.

\medskip

We want to start the presentation with two examples simpler than the one we use for the main algorithm, just to give some intuition on how these can be constructed. Our easiest example is when $a : I \to \R$ is the basic zig-zag centred at $2$, 
defined by
\begin{equation}\label{4b6} 
a(x) = 0 \text{ on } \R \sm [0,4] , \ 
a(x) = x \text{ on } [0,1], \ 
a(x) = 2-x \text{ on } [1,3], \ 
a(x) = x-4 \text{ on } [3,4].
\end{equation}
By direct computations, that we shall not copy here because in the end we need a smoother example, one can check that, 
for $s = t$ small enough, $\Delta \geq 0$ everywhere and $\Delta \geq C^{-1} t^2$ on a significant part of $I$.
Instead of a computation, we give here a picture of the graphs of $a$, $b$, and $t^{-2} \Delta$.
See Figure \ref{f1}. 
\begin{figure}[!h] 
  \centering 
 \includegraphics[width=8cm]{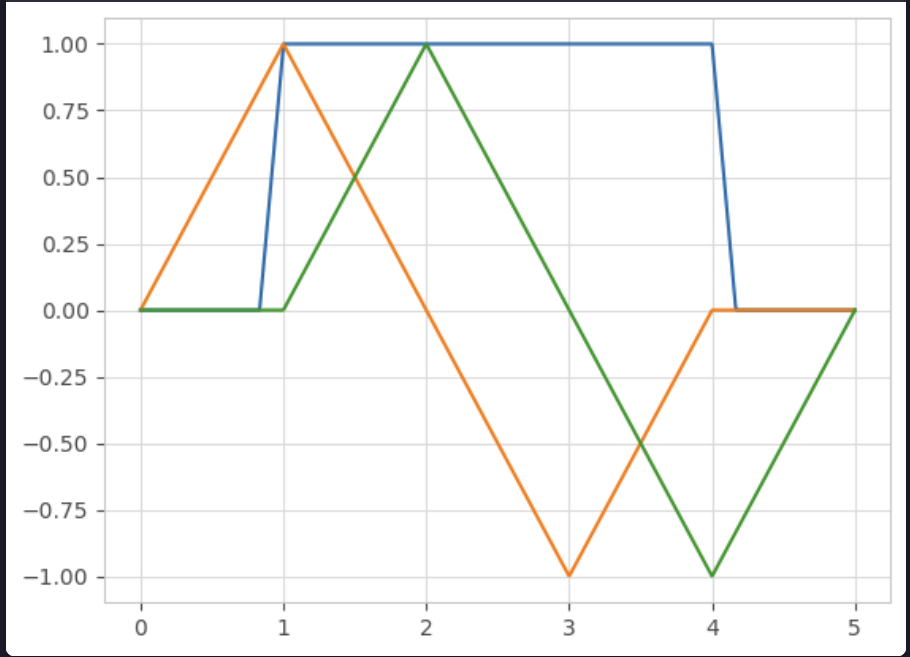} 
\caption{The output of a Python computation, for the functions $a$ (in yellow), $b$ (in green), and $t^{-2}\Delta$ (in blue). Here $t=\frac{1}{5}$; smaller values only make the blue slopes larger.} 
\label{f1}
\end{figure} 

The  reader may have noticed that $a$ and $b$ look vaguely like two truncated sine functions, and indeed the function
$a$ defined by
\begin{equation} \label{4b7}
a(x) = \1_{[0,4]} \sin\left(\frac{\pi x}{2}\right) 
\end{equation}
seems to works even better, due to a simple identity on the sine functions and their derivatives. 
Again we do not show any computations because we shall do a smooth cut-off version, 
and instead give the graphs of $a$, $b$, and $t^{-2} \Delta$ 
for two values of $t$.
See Figure \ref{f2}. 

\begin{figure}[!h] 
  \centering 
  \includegraphics[width=8cm]{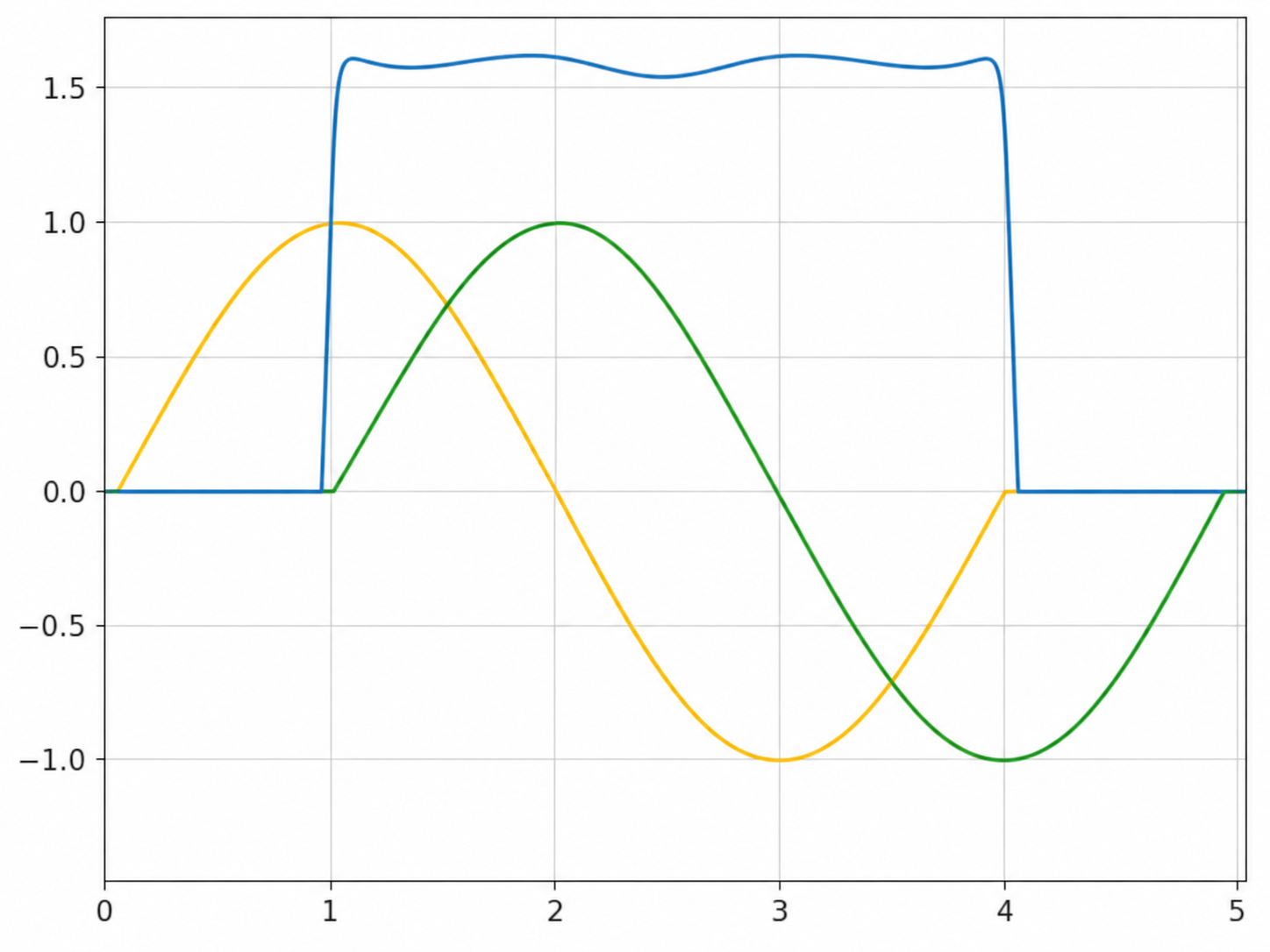} 
  \includegraphics[width=8cm]{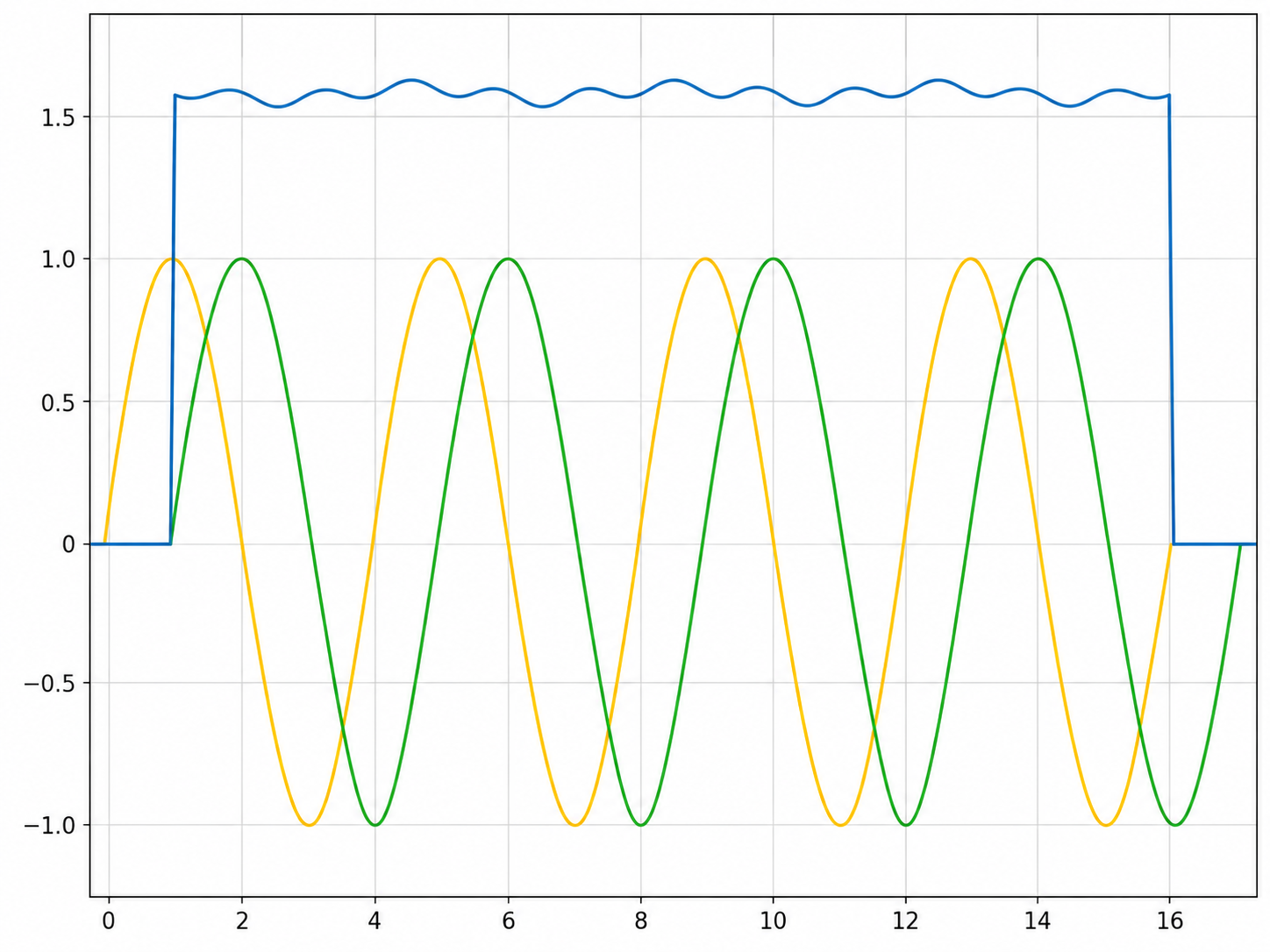} 
\caption{The output of a Python computation, for the truncated sine functions $a$
of \eqref{4b7} (in yellow), $b$ (in green), and $t^{-2}\Delta$ (in blue). Here $t=\frac{1}{20}$.
On the right, the same with a truncation over $4$ periods and $t = \frac{1}{15}$. 
} \label{f2}
\end{figure} 

Our final example is smoother, as needed, but a little more complicated. We choose a (smooth, non-negative) 
cut-off function $\theta$ symmetric with respect to $2$, such that 
\begin{equation}\label{4b8}
\theta(x) = 0 \text{ on } (-\infty, 0] \cup [4,+\infty), \ 
0 \leq \theta(x) \leq 1 \text{ on } \left[ 0, \frac{1}{4} \right] \cup \left[\frac{15}{4}, 4 \right], \ 
\theta(x) = 1 \text{ on } \left[ \frac{1}{4}, \frac{15}{4} \right],
\end{equation}
and then set 
\begin{equation}\label{4b9}
a(x) = \theta(x) \sin\left(\frac{\pi x}{2}\right)  \ \text{ and, as before } b(x) = a(x-1),
\end{equation}
hoping that he positivity of $\Delta$
(for $s=t$ small) will not suffer from that. See in Figure \ref{f3} the graphs of $a$ and $b$, computed with a $\theta$ based on a truncated polynomial (the polynomial is $35 x^4 - 84 x^5 + 70 x^6 - 20 x^7$, but this is not important: we just wanted something very smooth and increasing from $0$ to $1$ on the interval $[0, 1]$),
and two examples of $t^{-2} \Delta$ for this function. We added details of the left picture
to convince (ourselves) that $t^{-2} \Delta$ is now smoother and nothing fishy happens at the take-off 
(Figure \ref{f4}).

\begin{figure}[!h] 
  \centering 
 \includegraphics[width=8cm]{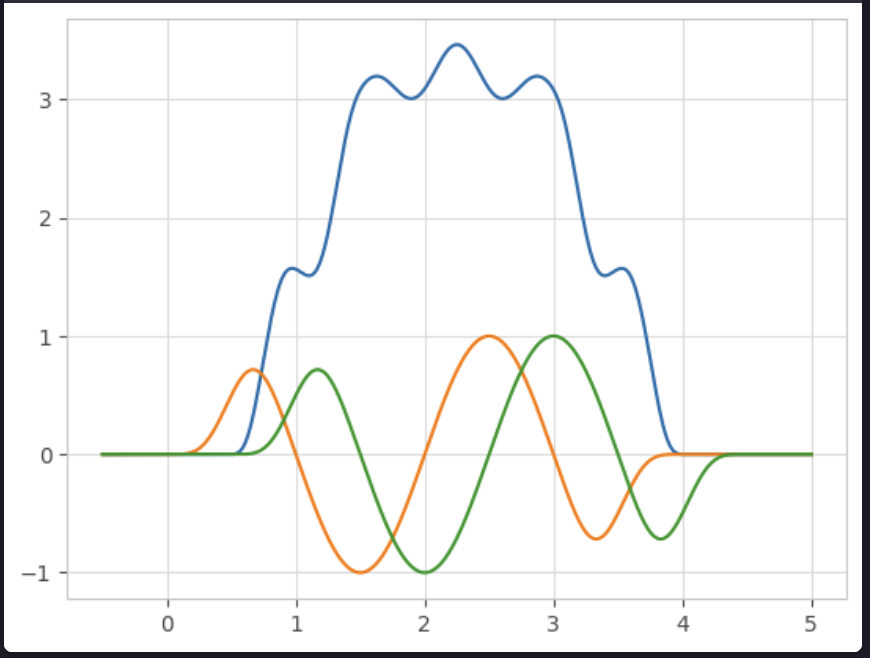} 
  \hskip2pt
  \includegraphics[width=8cm]{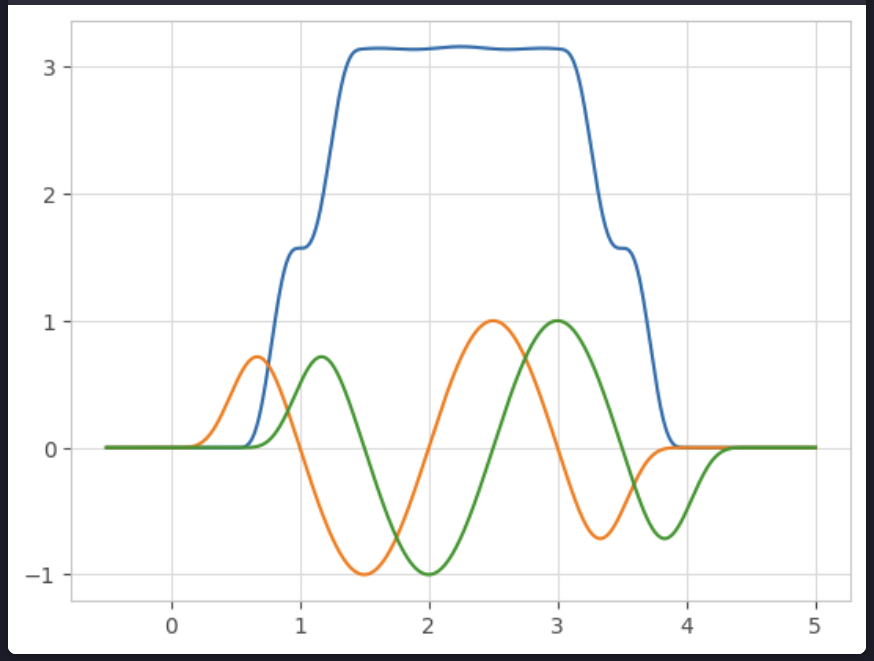} 
\caption{The output of Python computation, for the smoothly truncated sine functions $a$
of \eqref{4b9} (in yellow), $b$ (in green), and $t^{-2}\Delta$ (in blue). Here $t = \frac{1}{10}$ on the left, 
$t= \frac{1}{200}$ on the right.} \label{f3}
\end{figure} 

\begin{figure}[!h] 
  \centering 
 \includegraphics[width=6.8cm]{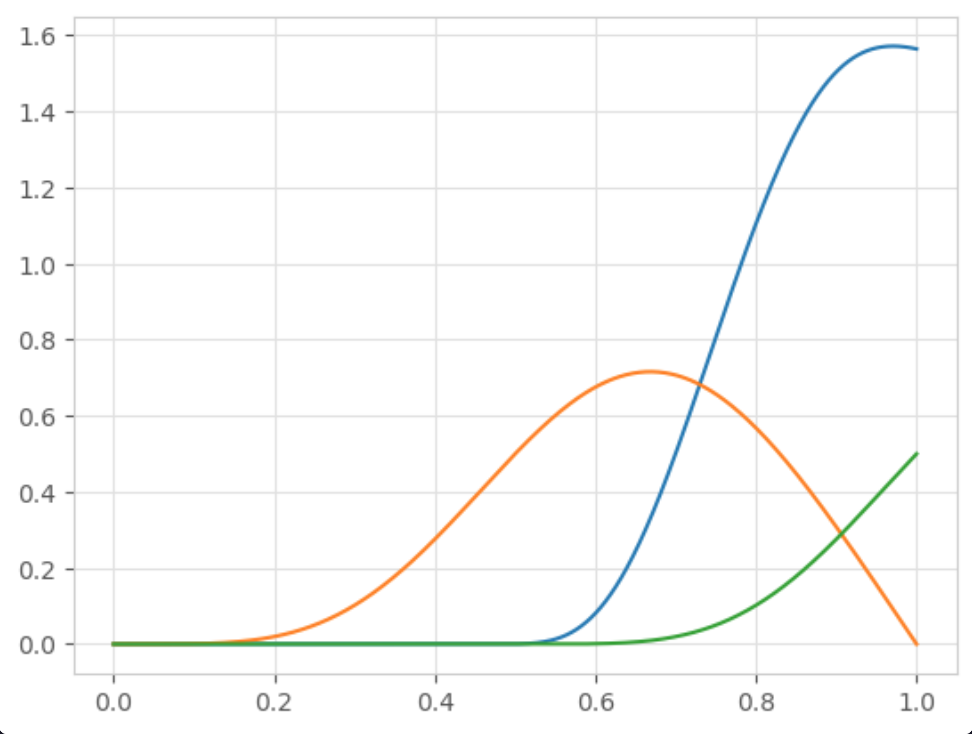} 
 \hskip5pt
  \includegraphics[width=7cm]{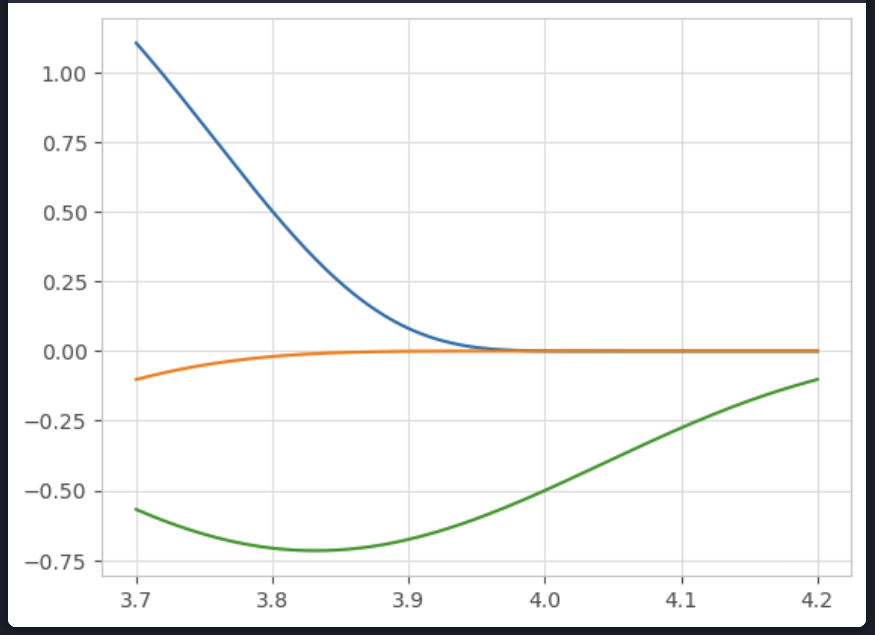} 
\caption{Details of the pictures above.} \label{f4}
\end{figure} 

\medskip

We now need more precise estimates on the sign and size of $\Delta$ in \eqref{4b5}. Before we start, we should observe that $sa$ and $tb$ are both antisymmetric with respect to the centers $2$ and $3$,
so $f$ and $g$ are easily constructible, and $R$ is in our group of interest. 
They are also as smooth as $\theta$.
We also decide that $s=t$ to simplify the computations.

We first observe that $f$ and $g$ coincide with the identity on $\R \sm [0,5]$ (because $a$ and $b$ vanish there), hence 
\begin{equation}\label{4b10}
\Delta = 0 \ \text{ on } \R \sm [0,5].
\end{equation}
So we just consider $[0,5]$ and try Taylor expansions. Write $f_t$ for $f$ and $g_t$ for $g$.
We are interested in 
\begin{equation}\label{4b11}
g_t(f_t(x)) = f_t(x) + t b(f_t(x)) = x + t a(x) + t b(x + t a(x)),
\end{equation}
\begin{equation}\label{4b12}
f_t(g_t(x)) = g_t(x) + t a(g_t(x)) = x + t b(x) + t  a(x + t b(x)),
\end{equation}
and then the difference
\begin{equation}\label{4b13}
\Delta_t(x) = g_t(f_t(x)) - f_t(g_t(x))
= t a(x) - t a(x+tb(x)) - t b(x) + t b(x + t a(x)). 
\end{equation}
Next we compute first derivatives 
\begin{equation}\label{4b14}
\partial_t (g_t \circ f_t(x)) = a(x) + b(x+t a(x)) + t   a(x) b'(x + t a(x)),
\end{equation}
and, by symmetry,
\begin{equation}\label{4b15}
\partial_t (f_t \circ g_t(x)) = b(x) + a(x+t b(x)) + t   b(x) a'(x + t b(x)).
\end{equation}
We take the difference and get that 
\begin{eqnarray}\label{4b16}
\partial_t (\Delta_t(x)) &=&  a(x) + b(x+t b(x)) + t   a(x) b'(x + t a(x)) 
\nn\\
&\,& - b(x) - a(x+t b(x)) - t  b(x) a'(x + t b(x)).
\end{eqnarray}

When $t=0$, we both get that $\Delta_t(x)=0$ and 
$\partial_t (\Delta_t(x))= t ab'-tba' = 0$ (as wanted). So $\Delta_t$ is of order $2$. 
More precisely, pick integers $r$ and $j$ (as large as we want, this will not cost us anything), and choose
$\theta$ of class $C^{r+j}$. Then, as 
as a function of $t$, $\Delta_t$  is a $C^r$ function, with values in the  class $C^{j}(I)$.

We divide by $t^2$, use the fact that $\Delta_t$ is of order $2$ in $C^j$, and get that 
\begin{equation}\label{4b17}
t^{-2} \Delta_t \ \text{ is a $C^r$ function of $t$, with values in the  class $C^{j}(I)$.}
\end{equation}
We prove below that 
$t^{-2} \Delta_t$ is also (uniformly) large in many interesting places, and even nonnegative everywhere. 

Taking one more derivative, we get
\begin{equation}\label{4b18}
\partial^2_{tt} (g_t \circ f_t(x)) =  a(x) b'(x+t a(x))  + a(x) b'(x + t a(x)) + t^2 a(x)^2 b''(x + t a(x)),
\end{equation}
compute this at $t=0$, which gives $2a(x) b'(x)$, and subtract the symmetric formula to get that
\begin{equation}\label{4b19}
F(x) : = \partial^2_{tt} \Delta_t(x)_{\vert t=0} = 2a(x) b'(x) -2 b(x) a'(x).
\end{equation}
We discuss the sign and size of this depending on the regions; the point will be that $F$ resembles a bump function (that does not depend on $t$). We start with good news.
On the interval $\left[\frac{5}{4}, \frac{15}{4}\right]$, $\theta(x) = \theta(x-1) = 1$, so 
$$
a(x) = \sin\left(\frac{\pi x}{2}\right) \ \text{ and } \ b(x) = \sin\left(\frac{\pi (x - 1)}{2}\right) = - \cos\left(\frac{\pi x}{2}\right),
$$
hence $a'(x) =\frac{\pi}{2} \cos\left(\frac{\pi x}{2}\right)$ and $b'(x) =  \frac{\pi}{2} \sin\left(\frac{\pi x}{2}\right)$, so
\begin{equation}\label{4b20}
F(x) 
= 2a(x) b'(x) -2 b(x) a'(x) = \pi  \ \text{ on } \left[\frac{5}{4}, \frac{15}{4}\right].
\end{equation}
In general, though, we only have
\begin{equation}\label{4b21}
a(x) = \theta(x)\sin\left(\frac{\pi x}{2}\right) \ \text{ and } \ b(x) = - \theta(x-1) \cos\left(\frac{\pi x}{2}\right),
\end{equation}
so we'll need to account for the variations of $\theta$ too. 
But we claim that 
\begin{equation}\label{4b22}
F(x) = 0 \ \text{ on } [0,1] \cup [4,5].
\end{equation}
Indeed, on the first interval, $\theta(x-1) = \theta'(x-1) = 0$ so $b = b' = 0$, while on the second one
$\theta(x)=\theta'(x)=0$ so $a = a' =0$. By \eqref{4b13}, we also have that $\Delta_t = 0$ 
on the slightly smaller intervals $[0,1-t]$ and $[4+t,5]$, where we shrink the intervals because in 
some terms we compute $a$ or $b$ at points $x+a(x)t$ or $x+b(x)t$. So some activity may exist a tiny bit before $x=1$ and after $x=4$.

We are left with the two $\left[1, \frac{5}{4}\right]$ and $\left[\frac{15}{4}, 4 \right]$. On the first interval, \eqref{4b21} yields
\begin{equation}\label{4b23}
a(x) = \sin\left(\frac{\pi x}{2}\right) \ \text{ and } \ b(x) = - \theta(x-1) \cos\left(\frac{\pi x}{2}\right),
\end{equation}
and so 
\begin{eqnarray}\label{4b24}
F(x) &=& 2a(x) b'(x) -2 b(x) a'(x)
\nn\\
&=& \sin\left(\frac{\pi x}{2}\right)\left[-2\theta'(x-1) \cos\left(\frac{\pi x}{2}\right) + \pi \theta(x-1) \sin\left(\frac{\pi x}{2}\right)\right]
+ \pi \theta(x-1) \cos^2\left(\frac{\pi x}{2}\right)
\nn\\
&=& \pi\theta(x-1) - 2\theta'(x-1) \cos\left(\frac{\pi x}{2}\right)\sin\left(\frac{\pi x}{2}\right)
= \pi \theta(x-1) - \theta'(x-1) \sin(\pi x).
\end{eqnarray}
We are not surprised (by \eqref{4b22}) that this vanishes at $1$ (and on its left). On the right,
$\theta(x-1)$ and $\theta'(x-1)$ are nonnegative, while  $\pi x$ is just bigger than $\pi$, so 
$\sin(\pi x)$ is negative and the second term is positive. Therefore we have
\begin{equation}\label{4b25}
F(x) \geq  \pi \theta(x-1) \geq 0 \ \text{ on }  \left[1, \frac{5}{4} \right].
\end{equation}

On the remaining interval $\left[\frac{15}{4}, 4\right]$,  \eqref{4b21} yields
\begin{equation}\label{4b26}
a(x) = \theta(x)\sin\left(\frac{\pi x}{2}\right) \ \text{ and } \ b(x) = -  \cos\left(\frac{\pi x}{2}\right),
\end{equation}
so 
\begin{eqnarray}\label{4b27}
F(x) &=& 2a(x) b'(x) -2 b(x) a'(x)
\nn\\
&=& \pi \theta(x) \sin^2\left(\frac{\pi x}{2}\right)
+ \cos\left(\frac{\pi x}{2}\right) \left[\pi \theta(x) \cos\left(\frac{\pi x}{2}\right) + 2\theta'(x) \sin\left(\frac{\pi x}{2}\right)\right]
\nn\\
&=& \pi\theta(x) + 2\theta'(x) \cos\left(\frac{\pi x}{2}\right)\sin\left(\frac{\pi x}{2}\right)
= \pi \theta(x) + \theta'(x) \sin(\pi x).
\end{eqnarray}
Now we care about $x$ slightly smaller than $4$, so $\pi x$ is slightly smaller than $4 \pi$,
the sine is negative, but so the derivative $\theta'(x)$ is. Therefore, again, the two terms have the same sign (positive) and 
\begin{equation}\label{4b28}
F(x) \geq  \pi \theta(x-1) \geq 0 \ \text{ on } \left[\frac{15}{4}, 4\right].
\end{equation}
So $F$ is a nice enough nonnegative bump function. What will really matter here is that,
by the Taylor formula,
\begin{equation}\label{4b29}
\Delta_t(x) = \frac{t^2}{2} F(x) + O(t^3),
\end{equation}
with a uniform $O(t^3)$ and a nonnegative bump function $F$ which is 
equal to $\pi$ on the central interval $\left[\frac{5}{4}, \frac{15}{4}\right]$. 

\ms
Let us also check that 
\begin{equation}\label{4b30}
\Delta_t(x) \geq 0  \text{ for all $x$, and $t>0$ small enough,}
\end{equation}
i.e., that there is no strange oscillation in the places where $F(x)$ is very small.
(Here we take $t > 0$ so that we don't have to distinguish more cases; we did not try to check $t < 0$.)

We know that $\Delta_t(x) = 0$  in the intervals $(-\infty, 1-t]$ and $[4+t, +\infty)$ (see \eqref{4b10} and the comment below \eqref{4b22}), on the middle interval
$\left[\frac{5}{4}, \frac{15}{4}\right]$, $F$ is sufficiently large, so \eqref{4b30} follows from \eqref{4b29}.
We are left with two small intervals near $1$ and $4$, where one of the two functions
$\theta(x), \theta(x-1)$ is $1$, and the other one may be small and vary.

We use \eqref{4b13} and the fundamental theorem of calculus to write
\begin{equation}\label{4b31}
\Delta_t(x)  = t a(x) - t a(x+tb(x)) - t b(x) + t b(x + t a(x)) = t^2(A+B)
\end{equation}
where 
\begin{equation}\label{4b32}
A = [a(x)-a(x+tb(x))]/t =  -a'(\xi) b(x)
\end{equation}
for some $\xi$ between $x$ and $x+tb(x)$ and 
\begin{equation}\label{4b33}
B = [b(x + t a(x)) - b(x)]/t = a(x) b'(\eta) 
\end{equation}
for some $\eta$ between $x$ and $x+ta(x)$.

We first  investigate near $1$, where $a$ and $b$ are given by \eqref{4b23} (even on the left of $1$),
so 
$$
A = -a'(\xi) b(x) =  \frac{\pi}{2} \cos\left(\frac{\pi \xi}{2}\right) \theta(x-1) \cos\left(\frac{\pi x}{2}\right)
$$
 and 
$$
B = a(x) b'(\eta) = \sin\left(\frac{\pi x}{2}\right) \left[\frac{\pi}{2} \theta(\eta-1) \sin\left(\frac{\pi \eta}{2}\right) -  \theta'(\eta-1) \cos\left(\frac{\pi \eta}{2}\right)\right].
$$
If $A \neq 0$, then $\theta(x-1) > 0$, then $x > 1$ and also $x + t a(x) \geq x > 1$
(recall that $a(x) = \sin\left(\frac{\pi x}{2}\right) \sim 1$), so both cosines are negative and $A \geq 0$.
As for $B$, if $B \neq 0$ then $\eta > 1$ and the term with $\theta'(\eta-1)$ is non-negative.
Also, the two sines are near $1$ and the first term is non-negative too.

Near $4$, $a$ and $b$ are given by \eqref{4b26}, so 
$$
A = -a'(\xi) b(x) = \left[\theta'(\xi) \sin\left(\frac{\pi \xi}{2}\right)  + \frac{\pi}{2} \theta(\xi) \cos\left(\frac{\pi \xi}{2}\right)\right] \cos\left(\frac{\pi x}{2}\right),
$$
$$
B= a(x) b'(\eta) = \frac{\pi}{2} \theta(x) \sin\left(\frac{\pi x}{2}\right) \sin\left(\frac{\pi \eta}{2}\right) .
$$
If $A \neq 0$, then $\xi < 4$, $\theta'(\xi) \leq 0$, $\sin\left(\frac{\pi \xi}{2}\right) \leq 0$
(because $\frac{\pi \xi}{2}$ is a little smaller than $2\pi$), so 
$\theta'(\xi) \sin\left(\frac{\pi \xi}{2}\right) + (\pi/2) \theta(\xi) \cos\left(\frac{\pi \xi}{2}\right) \geq 0$.
In addition, $\frac{\pi \xi}{2}$ lies close to $2\pi$, so $\cos\left(\frac{\pi x}{2}\right) \geq 0$ as well, and $A \geq 0$.
If $B \neq 0$, then $\theta(x) > 0$, hence $x < 4$, $a(x) = \theta(x) \sin(\pi x/2) \leq 0$,
$\eta \leq x \leq 4$, the two sine functions are negative (or $0$), and hence $B \geq 0$.
This completes our verification of \eqref{4b30}.

\medskip

Let us finally write down what this means for our commutator $R = R_t$ from
\eqref{4b3}. Since $R-I = g^{-1} \circ f^{-1}  \circ \Delta_t$, we get that for $t>0$ small enough,
\begin{equation}\label{4b34}
x \leq R(x) \leq x + C t^2  \ \text{ for } x \in \R
\end{equation}
and that on the central interval
\begin{equation}\label{4b35}
R_t(x) \geq x + ct^2   
\ \text{ on }  \Big[1 + \frac{3}{10}, 4 - \frac{3}{10} \Big],
\end{equation}
by \eqref{4b20} and \eqref{4b29}, where we removed a bit of the interval 
because $g^{-1} \circ f^{-1}$ may move points. If we wanted $R_t$ to move points in the other direction,
we could take the inverse, i.e., the other commutator $f^{-1} \circ g^{-1} \circ f \circ g $.

Also, if we set $\psi_t = R_t -I$ as usual, it is important to know that 
\begin{equation}\label{4b36}
t^{-2}\psi_t \text{ is a $C^r$ function of $t$, with values in the class $C^j(I)$,}
\end{equation}
by the proof of \eqref{4b17} (i.e., the fact that it stems from a $C^r$ function, Taylor of order $2$,
and the fact that $\psi_t$ and its $t$-derivative vanish at $0$).

\section{ Wavelets }
\label{S6}

As was explained before, our main step in targeting $f_\psi$ by constructible mappings with a small error will be the approximation of a given $\psi$ by a finite combination of wavelets. We have to choose this combination and the wavelet base very carefully however, due to the aforementioned issue with targeting something of non-zero average: we can do this, but this operation will not be cheap, so we need to avoid doing that often.

For a general background on wavelets and to be at ease with the discussion below, one can look to \cite{Da} and its references, or \cite{KL}, page 452 and further, for a discussion of the flexibility and advantages of bi-orthogonal wavelets. For a brief but concise overview aimed at a reader who likes Gabor frames more than wavelets, one can address to Chapter 10 of \cite{Gr}. We will also refer to Mallat's book \cite{Ma} in what follows.

Ideally, we would like to use antisymmetric wavelets, because of the special form of our antisymmetric
functions of the $\Phi_0(I)$ space. 
We would also like compactly supported wavelets, to keep all the computations (analysis and synthesis) in our initial interval $I$. 
Unfortunately, there is no orthogonal wavelet basis composed of antisymmetric wavelets, if we exclude the Haar system
which is not smooth enough for us. So the best solution for us seems to be compactly supported bi-orthogonal frames of wavelets, such as the Lemari\'e-Battle, \cite{B1}, \cite{B2} and \cite{LR},
or Cohen-Daubechies-Feauveau \cite{CDF} wavelets. Having a dual basis different from the original one is slightly scary, 
but we don't really care if we use one set of functions for analysing $\psi$, and a different one for reconstructing it.

\medskip

We now give the notation (which is pretty standard) and list the properties of the wavelets base we use, to keep things self-contained. Let us keep to the convention that we use Cohen-Daubechies-Feauveau wavelets, and though the properties we need are a little bit difficult to localize in the original paper \cite{CDF}, one can refer to Section $7.4$ in Chapter $7$ of \cite{Ma}, where these are nicely presented. 

Our wavelets will be based on a multiscale resolution, which in particular means that there is a (closed) space
$V_0 \subset L^2$ at scale $1$ (chosen carefully), with a basis $\{ \varphi_i \}, i \in I_0$, and that for functions $u \in V_0$ 
there is a decomposition
\begin{equation} \label{5a1}
u = \sum_{i \in I_0} \lambda_i \varphi_i,
\end{equation}
where the coefficients $\lambda_i$ can be computed as
\begin{equation}\label{5a2}
\lambda_i = \lambda_i(u) = \int u(x) \ol\varphi_i dx.
\end{equation} 
For traditional orthogonal wavelets with complex coefficients, $\ol\varphi_i$ would be the complex conjugate
of $\varphi$, but here we will use real-valued biorthogonal wavelets, and the $\ol\varphi_i$ 
are just another set of functions (the dual basis). Our two bases are obtained by translations from 
``scaling functions'' $\varphi$ and $\ol\varphi$, in the sense that 
for each $i \in I_0$, there is a point $x_i$ in a grid such that 
\begin{equation} \label{5a3}
\varphi_i(x) = \varphi(x-x_i) \ \text{ and } \ \ol\varphi_i(x) = \ol\varphi(x-x_i).
\end{equation}
In addition, it is traditional that the $x_i$ are the points of $\bZ$, but for us it will be more convenient to work at 
a scale $\tau < 1$ and use the points $x_i \in \tau \bZ$. 
This amounts to conjugating everything by a dilation, or starting from a space 
$V_{k_0}$ with a positive $k_0$ in the usual construction. We shall choose $\tau$ soon, depending on our wavelet basis. We can ask for our scaling function $\varphi$ to be smooth enough and compactly supported, so that in particular
\begin{equation} \label{5a4}
\varphi \in C^k(\R) \ \text{ and is supported in } [-K\tau, K\tau],
\end{equation}
where $K$ depends only on $k$ and our choice of wavelets. We are also allowed to require that
\begin{equation} \label{5a5}
\ol\varphi  \text{ lies in some $C^\beta$, $\beta>0$, and is supported in } [-K\tau, K\tau].
\end{equation}
Now it is part of life that apparently $\ol\varphi$ is not as smooth as $\varphi$, which is a little disturbing, 
but this will not matter much. And we choose our initial scale $\tau$ such that 
\begin{equation} \label{5a6}
\tau K \leq 10^{-2}
\end{equation}
(to be completely safe). Recall from \eqref{2e8} that our target $\psi_1$ after the curtain trick is supported on $\left[\frac{1}{30}, 1 - \frac{1}{30}\right]$,
so when we decompose it, we may as well restrict to the set $I_0$ if indices $i$ such that 
$x_i \in \left[\frac{1}{50}, 1 - \frac{1}{50}\right]$,
and then the pieces will still be supported in $[10^{-2}, 1-10^{-2}]$. 

To wrap up the description of the top layer of the wavelet decomposition, let us remark that we are allowed to choose an even function $\varphi$ with all the smoothness and localisation properties above, see pp 312-313 of \cite{Ma} 
(this does not seem to help us though).

In our multiscale resolution, there are also further spaces 
$V_j = \big\{ u \in L^2 \, ; \, u(2^j \cdot) \in V_0 \big\}$ at smaller scales, and the spaces $W_j$, the orthogonal 
complements of $V_{j}$ in $V_{j+1}$. Bases of $W_j$ are generated by a ``mother function'' $\theta$, 
while some dual function $\ol \theta$ generates our dual bases. 
They enjoy the following properties. As before, 
\begin{equation} \label{5a7}
\theta\in C^{k+1}(\R), \ol\theta \in L^\infty(\R), \ \text{ both are supported in } [-K\tau, K\tau]
\end{equation}
(we added a derivative because we will need it later), and we can ask that (see again pp 312-313 of \cite{Ma})
\begin{equation} \label{5a8}
\theta \text{ is odd.} 
\end{equation}
Concerning $\ol\theta$, what replaces the missing smoothness and turns out to be the right notion here, is the vanishing moments condition
\begin{equation} \label{5a9}
\int \ol\theta(x) x^m dx= 0  \text{ for } 0 \leq m \leq k.
\end{equation}
So, the bases of $W_j$ are denoted by 
$\{ \theta_i \}$, and the dual ones by $\{ \ol \theta_i \}$, and we have
\begin{equation} \label{5a10}
\theta_i(x) = 2^{j/2}\theta(2^{j}(x-x_i)) \ \text{ and } \ \ol\theta_i(x) = 2^{j/2} \ol\theta(2^{j}(x-x_i)),
\end{equation}
where we normalized in $L^2$, and the centers $x_j$ lie in the set $I(j) = \tau 2^{-j} \bZ$.
In fact, we only take the $x_i$ such that $\theta_i$ and $\ol\theta_i$ are supported well inside $[0,1]$, 
for the same reason as above. 

Call $P_j$ and $Q_j$ the 
projections on $V_j$ and $W_j$ respectively for $j \ge 0$. For general $u \in L^2$, we can still compute the coefficients $\lambda_i$ as above, and the
right-hand side of \eqref{5a1} is now $P_0(u)$ (rather than $u$ itself). Concerning the $\theta_i$ functions,  
we have that
\begin{equation} \label{5a11}
Q_j(u) = \sum_{i \in I(j)} \alpha_i \theta_i,
\end{equation}
with 
\begin{equation} \label{5a12}
\alpha_i = \int u \ol \theta_i,
\end{equation}
and altogether
\begin{equation} \label{5a13}
u = P_0(u) + \sum_{j \geq 0} Q_j(u) 
= \sum_{i \in I_0} \lambda_i \varphi_i + \sum_{j \geq 0} \sum_{i \in I(j)} \alpha_i \theta_i.
\end{equation}
Usually, we have convergence in $L^2$, but for our functions, we have uniform convergence, and even better, 
as we will see now.

\medskip

We are interested in functions $u \in C^k$, with compact support in $[0,1]$, and we want to estimate the size of $\alpha_i$ for some $i \in I(j)$, in terms of $N_k(u)$. 
Recall that $\ol\theta_i$ is supported on an interval $L_i$ 
of length $2^{-j} K \tau$, and on this interval there is a
polynomial $p$, of degree at most $k$, such that $|u-p| \leq C N_k(u) 2^{-kj}$ 
(this is called Jackson's theorem;  we almost proved it for \eqref{2e3}). 
Then
\begin{equation} \label{5a14}
|\alpha_i| = \Big|\int_{L_i}  u \ol \theta_i \Big|  = \Big|\int_{L_i}  (u-p) \ol\theta_i \Big| 
\leq C N_k(u) 2^{-j k} |L_i| ||\theta_i||_\infty \leq C N_k(u) 2^{-j/2} 2^{-kj},
\end{equation}
by the vanishing moment condition and the normalization \eqref{5a10}. 
Since the sum in \eqref{5a11} is locally finite (with bounds),
we deduce from \eqref{5a10} that 
\begin{equation} \label{5a15}
\Big| Q_j(u) \Big| \leq C N_k(u) 2^{-kj}
\end{equation}
and even more, using the regularity of $\theta$, that 
\begin{equation} \label{5a16}
N_k(Q_j(u)) \leq C N_k(u).
\end{equation}
For the initial level, the same computations, but simpler, yield
$\lambda_i \leq C N_0(u)$, and then
\begin{equation} \label{5a17}
N_k(P_0(u)) \leq C N_0(u)
\end{equation}
since the $\varphi_i$ are smooth enough. Finally, for a more general $P_j$,
\begin{equation} \label{5a18}
N_k(P_{j+1}(u)) \leq N_k(P_0) + \sum_{\ell \leq j} N_k(Q_\ell(u)) \leq C N_0(u) + C j N_k(u).
\end{equation}
This is slightly diverging, but still under control.

\ms

We intend to apply all this to $u = \psi_1$, where $\psi_1$ is our first renewed target from Section \ref{S3}.
In fact, we will regroup our sum $\sum_i \alpha_i \theta_i$ in more terms than just the $Q_j$, because for each 
$j$ we will cut the set $I(j)$ into less than $C$ sets $I(j, \ell)$ where all the $\theta_i$ are supported in disjoint intervals, so that
the corresponding sum lies in the larger class $\Psi([0,1])$ of constructible mappings defined near \eqref{1a13}.
The same proof as above still gives
\begin{equation} \label{5a19}
\Big|  \sum_{i \in I(j,\ell)} \alpha_i \theta_j \Big| \leq C N_k(u) 2^{-kj}
\end{equation}
as in \eqref{5a15}, and 
\begin{equation} \label{5a20}
N_k\left(\sum_{i \in I(j,\ell)} \alpha_i \theta_j \right) \leq C N_k(u).
\end{equation}
Using wavelets with one more derivative, we can even have
\begin{equation} \label{5a21}
N_{k+1}\left(\sum_{i \in I(j,\ell)} \alpha_i \theta_j \right) \leq C  2^j N_k(u)
\end{equation}
(an estimate that we intend to use with parsimony, only when it is compensated by some smaller term).
Also, when we sum \eqref{5a15} for the rest of the series, we get
\begin{equation} \label{5a22}
|u - P_{j+1}(u)| \leq \sum_{\ell > j} |Q_j(u)| \leq C N_k(u) 2^{-kj}.
\end{equation}
We will typically stop the series for a $j$ large enough (called $N$ later) for this to be very small.

\ms

\begin{rem}
Let us say a few words of how we could deal with Meyer wavelets. First, we would decompose $\psi'$ rather than $\psi$,
into symmetric wavelets. Then we would actually truncate both the coefficients and the wavelets, so that the coefficients $\alpha_i$
are still computed by integrating over a small interval, and the wavelets $\theta_i$ are truncated to be compactly supported. This 
would not disturb the estimates on the sums (we could even truncate our wavelets so that they have $k$ vanishing moments too),
and the price to pay is that we compute $Q_j$ with an additional error, which we can make small by truncating $\theta$ only on a large enough
ball. We prefer not to do this; thus, we use pre-arranged compactly supported wavelets. 

Also, note that we do not really need algebraic formulas for the wavelet decomposition, 
just a decomposition of some projection $P_0(u)$ on some space of co-dimension at most $C$,
into compactly supported antisymmetric normalized functions, with slow divergence of the $N_k$ norm, 
and fast convergence in the sup norm.
Things like truncated Littlewood-Paley decompositions (or the Calder\'on formula) should work as well.
\end{rem}

\section{Projection on $V_0$}
\label{S7}

Life would be easier if our first revised target $\psi_1$ was directly a sum of odd wavelets, but it has a projection
on $V_0$ which needs to be taken care of. The goal for this section is to replace $\psi_1$ with a new target, say, 
$\psi_2$, which is directly orthogonal to $V_0$, or almost, just as we would have liked for $\psi_1$. 
To achieve that, we will compose with a constructible mapping built from the commutators in Section \ref{S5}.

In other words, we look for a mapping $\psi = \psi_{proj}$ such that 
for the new target $\psi_2$, satisfying $f_{\psi_1} = f_{\psi_2} \circ f_{\psi_{proj}}$, 
or equivalently $\psi_2 = f_{\psi_1} \circ f_{\psi_{proj}}^{-1} - I$,
we have the following property. For  each $i \in I_0$ (the set of indices for the bases $\{ \varphi_i \}$
and $\{ \ol\varphi_i \}$, where we removed the useless ones that are supported too close to the complement of $I$),
\begin{equation}\label{6a1}
\lambda_i(\psi_2) : = \int \psi_2(x) \ol\varphi_i(x) dx = 0.  
\end{equation}
In fact, we will not exactly get this, but only make sure that $\lambda_i(\psi_2)$ is sufficiently small 
and show that
\begin{equation} \label{6a2}
   \lambda_i(\psi_{proj}) = \lambda_i(\psi_1) + O(N_0(\psi_1)^{3/2}).
\end{equation}
Since the symbol $\psi_2$ for $f_{\psi_1} \circ f_{\psi_{proj}}^{-1}$ is close to $\psi_1 -  \psi_{proj}$,
with \eqref{6a2} we will not be far from \eqref{6a1}.

\ms
We now construct our mappings $\psi = \psi_{proj}$ and $\psi_2$. 
We will continue 
to denote by $\wt \psi$ the function such that $I + \wt \psi = f_\psi^{-1}$
(i.e., the inverse of $\psi$ for $\ast$), and since 
\begin{equation}\label{6a2a}
\psi_2(x) = \psi_1 \ast \wt \psi(x) = \wt\psi(x) + \psi_1(x+\wt \psi(x)) 
\end{equation}
by \eqref{3e2}, \eqref{6a1} would amount 
to making sure that for $i \in I_0$,
\begin{equation}\label{6a3}
\int \big[\wt\psi(x) + \psi_1(f_\psi^{-1}(x)) \big] \ol\varphi_i(x) dx = 0.
\end{equation}
However, this will be difficult to solve, so we try instead to solve the simpler equations
\begin{equation}\label{6a4}
\int \wt\psi(x) \ol\varphi_i(x) dx = b_i \ \text{ with }
b_i = - \int \psi_1(x) \ol\varphi_i(x) dx
\end{equation}
where we make things linear by forgetting the composition with $f_{\psi}^{-1}$.
Let us discretize, and cut $[0,1]$ into $n+2$ intervals of length
$\ell = (n+2)^{-1}$. We remove the first and last intervals, because (if $n$ is large enough)
they will not meet the support of any of the wavelets $\varphi_i, \ol\varphi_i$, $i \in I_0$
that we use, and we are left with intervals $L_j$, $1 \leq j \leq n$; 
we don't care if there are still too many. We will soon choose $n$
large enough for some matrix to be invertible, but not too large.

We have a model for the function $\wt\psi$ that we intend to choose. We will choose the restriction
of $\wt\psi$ on each $L_j$, always with the method suggested in Section \ref{S5}, but
modulo the fact that we conjugate by an affine mapping that sends $L_j$ to the 
interval $[0,5]$.  
In each $L_j$, we choose a parameter $t = t_j$, and apply
the construction of Section \ref{S5} to find functions $f_t$, $g_t$, $R_t$, etc. When $t > 0$, we
use the same conventions as in that section, 
and when $t < 0$, let us use the convention that 
we exchange $f$ and $g$ (or equivalently replace $R$ by its inverse), with the effect of replacing $\Delta_t$
by $-\Delta_t$. With these conventions, \eqref{4b29} becomes
\begin{equation}\label{6a5}
\Delta_t(x) = \frac{t^2}{2} \, {\rm sign}(t) \, F(x) + O(t^3),
\end{equation}
where $F$ (which we shall also call $F_j$ because it now 
depends on $L_j$ through a translation)
is the function of \eqref{4b19}. Thus the $F_j$ have a fixed size and shape, and we will end up choosing
the $t_j$ small. Now we are really interested in taking 
\begin{equation}\label{6a6}
\wt \psi = R_t - I \text{ on } L_j,
\end{equation}
which we define, interval by interval, with the value $t_j$, and also with the convention
that we replace $R_t$ by its inverse when we switch to $t < 0$. 
With this convention, composing with
$f_t^{-1}$ and $g_t^{-1}$ as above \eqref{4b34} transforms \eqref{6a5} into 
\begin{equation}\label{6a7}
\wt \psi(x) = \frac{t_j^2}{2} \, {\rm sign}(t_j) \, F_j(x) + O(t_j^3)
\ \text{ on } L_j . 
\end{equation}
Note that on the right-hand side the asymptotics in $t_j$ stands since it is preserved by the compositions 
with $f_{t_j}^{-1}$ and $g_{t_j}^{-1}$. As usual, one of the main points of choosing these functions is 
that $\wt\psi$ is constructible. 
Set $\mu_j = \frac{t_j^2}{2} \, {\rm sign}(t_j)$; 
we want to use linear algebra to choose the $\mu_j$.
The left-hand side of \eqref{6a4} is
\begin{equation}\label{6a8}
\int \wt\psi(x) \ol\varphi_i(x) dx = \sum_j \mu_j \int_{L_j} F_j \ol\varphi_i + \cE,
\end{equation}
where the error $\cE$ is less than $C \sup_j |t_j|^{3} = C \sup_j |\mu_j|^{3/2}$.
We intend to forget about $\cE$, and solve the equation $AM = B$, where 
$B = \{ b_i \}$ is a column matrix of size $p := \sharp I_0$, 
$M$ is the column matrix of size $n$ composed of the $\mu_j$,
and $A$ is the $p \times n$ matrix with coefficients $a_{i,j} = \int_{L_j} F_j \ol\varphi_i$.
Thus $A$ is bounded, with bounds that depend only on $n$ and our choice of basis.

If we want to solve this, we need to know that the $\ol \varphi_i$ have enough independence, and this is 
given in a quantitative way by the fact that our wavelets come from a bi-orthogonal system. That is, the definition 
of a bi-orthogonal system says that for $i, r \in I_0$, 
\begin{equation} \label{6a9}
 \int_{[0,1]} \ol\varphi_i(x) \varphi_r(x) dx = \delta_{i,r} 
\end{equation}
(the Kronecker symbol). 

Let us first guess an approximate solution, by pretending that $\ol\varphi_i$ and $\varphi_r$ are all constant on every
$L_j$. Call $y_j$ the center of $L_j$, and also set $\rho = \fint_{L_j} F_j$; this is a positive constant, by Section \ref{S5} and Definition \eqref{6a7}. In this special case, 
$a_{i,j} = \rho\ell \ol\varphi_i(y_j)$, and \eqref{6a9} says that 
\begin{equation} \label{6a10}
\ell  \sum_j \ol\varphi_i(y_j) \varphi_r(y_j) = \delta_{i,r},
\end{equation}
which we can write more concisely as $\ell  E_0 F_0 =I$, where $E_0$ is the $p \times n$ matrix with coefficients 
$\ol \varphi_i(y_j)$ and $F_0$ is the $n \times p$ matrix with coefficients $\varphi_r(y_j)$. In addition,
in this special case $A = \rho \ell E_0$. So we want to solve $\rho \ell E_0 M = B$, 
we try with $M = \lambda F_0 B$, notice that we then want $\lambda\rho \ell E_0 F_0 B = B$, take $\lambda = \rho^{-1}$,
and then obtain a solution immediately.

In the general case, we can still try to find a solution $M = \rho^{-1} F_0 Y$, and see what 
our equation $AM = B$, or equivalently $\rho^{-1} AF_0 Y=B$ becomes. This is still a linear equation (with $p$ equations
and $p$ variables), whose matrix $\rho^{-1} A F_0$ has a more complicated structure, and obviously
it will help to estimate the norm of $S = \rho^{-1} A F_0 - I$. The coefficients are 
\begin{equation} \label{6a11}
\gamma_{i,r} = -\delta_{i,r} + \rho^{-1} \sum_j a_{i,j} \varphi_r(y_j)
= - \delta_{i,r} + \rho^{-1} \sum_j \varphi_r(y_j) \int_{L_j} F_j(x) \ol\varphi_i(x)dx .
\end{equation}
But by \eqref{6a10}, 
\begin{equation} \label{6a12}
\delta_{i,r} = \ell  \sum_j \ol\varphi_i(y_j) \varphi_r(y_j)
= \rho^{-1} \sum_j \int_{L_j} \ol\varphi_i(y_j) \varphi_r(y_j) F_j(x) dx, 
\end{equation}
so 
\begin{equation} \label{6a13}
|\rho\gamma_{i,r}| \leq \sum_j \int_{L_j} |F_j(x) \varphi_r(y_j)| \big| \ol\varphi_i(x) - \ol \varphi_i(y_j)|
\leq C ||\ol\varphi_i||_{C^\beta} \ell^{\beta} \leq C \ell^\beta,
\end{equation}
where we use the fact that our wavelets $\ol\varphi_i$ are all $C^\beta$ for some $\beta > 0$ 
(which depends on $k$, our regularity); as stated in \eqref{5a7}.
We don't need to be careful about the number $p$ of wavelets, 
because it is fixed, but we need to be careful about how our estimates above depend on 
$n$ and $\ell = (n+2)^{-1}$, to be chosen very soon.

At this point we can choose $n$ and $\ell$ so that $||| S ||| \leq 10^{-1}$, say, so that $\rho^{-1} AF_0$ is invertible,
with norm at most $2$, and the solution that we just produced to the linear version of \eqref{6a8} is 
$Y = [\rho^{-1} AF_0]^{-1} B$. That is, for the choice $M$ of coefficients $\mu_j$
(and then the ensuing choice of $\wt \psi$), we only get that the wavelet coefficients of $\wt\psi$ satisfy
\begin{equation} \label{6a14}
\int \wt\psi(x) \ol\varphi_i(x) dx = b_i + O(\mu_+^{3/2}),
\end{equation}
where 
\begin{equation} \label{6b15}
\mu_+ = \sup_{j} |\mu_j| \leq C |M| \leq C |B| \leq C N_0(\psi_1)
\end{equation}
(recall from \eqref{6a4} that $b_i$ is a wavelet coefficient of $\psi_1$).
This is not exactly \eqref{6a4}, because we have an error $O(\mu_+^{3/2})$, but this is not so bad. 
The reader no longer needs to worry about the size $n$ (which may make our function $\wt\psi$ look a little wilder), 
because $n$ was chosen above, depending only on our choice of wavelets.

Recall that we were really interested in the non-linear equation \eqref{6a3} and,
by the fundamental theorem of calculus, we make an error of size at most
\begin{equation} \label{6a15b}
\int |\psi_1(f_\psi^{-1}(x)) - \psi_1| \, |\ol\varphi_i(x)| dx \leq C N_1(\psi_1) N_0(\wt \psi) \leq C N_1(\psi_1) N_0(\psi_1).
\end{equation}
Altogether, the first wavelet coefficients of our new target $\psi_2$ are 
\begin{eqnarray} 
\Big| \int \psi_2(x) \ol\varphi_i(x) dx \Big| 
&=& \Big|\int \big[\wt\psi(x) + \psi_1(f_\psi^{-1}(x)) \big] \ol\varphi_i(x) dx \Big|
\nn\\ 
&\,& \hskip-50pt \leq \; \Big|\int \big[\wt\psi(x) + \psi_1(x) \ol\varphi_i(x) dx \Big| + C N_1(\psi_1) N_0(\psi_1)
\leq  C N_0(\psi_1)^{3/2} + C N_1(\psi_1) N_0(\psi_1)
\end{eqnarray}
(see \eqref{6a3} and above, and then \eqref{6a14} and \eqref{6b15}).
In other words, we now have that for the projection of $\psi_2$ on the space $V_0$, 
\begin{equation} \label{6a16}
N_k(P_0(\psi_2)) \leq C  N_0( P_0(\psi_2)) \leq C N_1(\psi_1) N_0(\psi_1) + C N_0(\psi_1)^{3/2} \leq C N_k(\psi_1)^{3/2},
\end{equation}
where the first inequality is \eqref{5a17}. 
Recall from \eqref{6a2a} that $\psi_2 = \psi_1 \ast \wt\psi = \psi_1 \ast \wt\psi_{proj}$
as in \eqref{3e12}, so by the first line of \eqref{3e13} and then \eqref{6a7} and \eqref{6b15} 
\begin{equation} \label{6a17}
N_k(\psi_2) \leq N_k(\psi_1) +  N_k(\wt \psi) + C N_k(\psi_1) N_k(\wt \psi)
\leq N_k(\psi_1) + C \mu_+
\leq C N_k(\psi_1) + C N_0(\psi_1) \leq C N_k(\psi_1),
\end{equation}
(recall that $\psi_1$ is supported in $I$).
So, by subtraction, for the odd wavelets part as well, 
\begin{equation} \label{6a18}
N_k(\psi_2 - P_0(\psi_2)) \leq C N_k(\psi_1).
\end{equation}
This completes the construction of our almost projected target $\psi_2$.

\ms
At this point, an important remark is due, on the norms of some of the elementary target mappings 
that we intend to use in our construction.
Observe that we have good estimates for $\wt \psi$ and our next target $\psi_2$, but worse estimates 
for the four pieces $f_t$, $g_t$, $f_t^{-1}$, and $g_t^{-1}$ that compose them, as they were obtained 
from Section \ref{S5} with a parameter $t_j$ such that $t_j^2= \pm \mu_j$. That is 
\begin{equation} \label{6b20}
N_k(f_t) + N_k(g_t) + N_k(f_t^{-1}) + N_k(g_t^{-1}) \leq C N_k(\psi_1)^{1/2},
\end{equation}
by \eqref{6b15} and the definition of $f_t$ and $g_t$.
Essentially, this is the reason why we will 
lose (almost) half of our derivatives in Theorem \ref{t1}.

\section{We remove odd wavelets and prove Theorem \ref{t1a9}} 
\label{S8}
\subsection{Removing the odd wavelets}

In this subsection we continue eliminating the coefficients in a wavelet decomposition of our target function. We will not use any estimates on those coefficients directly, though an eager reader may derive those from \eqref{7a23bis} 
 and get again a quantitative version of
$$
\alpha_i(\psi_3) = \alpha_i(\psi_2) + \overline{o}(\alpha_i(\psi_2)),
$$
calling $\psi_3$ our new target and using the notation of Section \ref{S6}. The main result of the (sub)section will be the estimate \eqref{1a17} on the $L^\infty$ norm of the difference $f - f_{\psi_3}$ 
that shows up in the conclusion of Theorem \ref{t1a9}, along with the desired estimate on $N_k(\psi_3)$.

\medskip

We start with the new target $\psi_2$ from the last section, and we add a bit of notation to better describe 
the decomposition of $\psi_2$. 
We write 
\begin{equation} \label{7a1}
\psi_2 = \eta + \xi = \eta + \xi_N + \xi_\infty,  \ \text{ with } \eta = P_0(\psi_2) 
\, \text{ and } \xi_N = P_{N+1}(\xi),
\end{equation}
where $\eta$ is small by \eqref{6a16}, the large integer $N$ will be chosen later so that $\xi_\infty$ is small enough, 
and we see $\xi_N$ as the main part of $\xi$
(and of $\psi_2$). We write
\begin{equation} \label{7a2}
\xi_N = \sum_{n=0}^N Q_n(\xi)
\end{equation}
the decomposition into pieces $Q_n(\xi)$ that contain wavelets of a given scale $n$, and further decompose each 
$Q_n(\xi)$ into a sum of at most $C$ terms $\xi_{n,j}$, $j \in I(n)$, where $\xi_{n,j}$ is a linear combination of wavelets 
$\sum_i \lambda_i\theta_i$, every $\theta_i$ in the sum is of generation $n$ and supported in some interval $J_i$,
and the $J_i$ are disjoint, and even at distances at least $2^{-n}$ from each other. 
This is why we need about $C$ families of $\theta_i$, and hence $C$ terms $j \in I(n)$ for each $n$. 
The point is roughly that 
\begin{equation} \label{7a3}
\text{each of the pieces $\xi_{n,j}$ is constructible, because it lies in $\Psi(J)$ for some $J \subset\subset[0,1]$.}
\end{equation}
The antisymmetry of our wavelets $\theta_i$ is important here, of course. But the truth is that we shall not
compose the $\xi_{n,j}$ themselves, but slightly different functions $\zeta_{n,j}$, $j \in I(n)$, 
obtained from the $\xi_{n,j}$
by affine transforms that will be explained soon, and which also satisfy the property \eqref{7a3}.

Let us renumber all the $\xi_{n,j}$ and the corresponding $\zeta_j$ as the $\xi_j$ and $\zeta_j$, 
$1 \leq j \leq j_N$, where $j_N \leq C N$, 
and for all $j$ let us denote by $n(j)$ the generation of wavelets that $j$ belongs to. 
We number the $\xi_j$ and $\zeta_j$ so that $n(j)$ is non-decreasing. 
The order will be important, because we will need to compose the functions
\begin{equation} \label{7a4}
f_j = f_{\zeta_j} = Id + \zeta_j,
\end{equation}
and composition is not commutative. We call 
\begin{equation} \label{7a5}
F_j = f_j \circ \ldots \circ f_1
\end{equation}
the partial compositions (which will be needed to compute differences). We want to evaluate things like $F_{j_N}-Id$
(the symbol of our last mapping), and more generally
\begin{eqnarray} \label{7a6}
F_j - Id &=& f_j \circ \ldots f_1 - Id = (f_j - Id) \circ F_{j-1} + F_{j-1} - Id = \ldots 
\nn\\
&=& \sum_{1 \leq m \leq j} (f_m - Id) \circ F_{m-1}
= \sum_{1 \leq m \leq j} \zeta_m \circ F_{m-1}.
\end{eqnarray}
The symbol $m$ is usually reserved for the integer part of $k/2$, but we use it differently in this section (in which we will have no interest in the parameter $k/2$). We would like $F_{j_N}- Id$ to be close to $\xi_N$, or equivalently $\delta$ to be small, where 
\begin{equation} \label{7a7}
\delta : = (F_{j_N} - Id) - \xi_N = \sum_{1 \leq m \leq j_N} \zeta_m \circ F_{m-1} - \xi_N
= \sum_{1 \leq m \leq j_N} (\zeta_m \circ F_{m-1} - \xi_{m}).
\end{equation}
So we are interested in the norms of
\begin{equation} \label{7a8}
\delta_m = \zeta_m \circ F_{m-1} - \xi_{m}.
\end{equation}
Of course taking $\zeta_m = \xi_{m} \circ  F_{m-1}^{-1}$ would sound nice, but it does not satisfy the 
antisymmetry condition required because we want to make sure that  
\begin{equation} \label{7a9}
\text{$\zeta_m$ is constructible, because it lies in $\Psi(J)$ for some $J \subset\subset[0,1]$.}
\end{equation}
So we will let $\zeta_m$ be a sum of translated and dilated wavelets, but not directly $\zeta_m = \xi_{m}$,
because we can make $\delta_m$ smaller that way. Recall that 
\begin{equation} \label{7a10}
\xi_{m} = \sum_i \alpha_i \theta_i
\end{equation}
with different wavelets $\theta_j$ that live in disjoint intervals $J_i$. We want to take
\begin{equation} \label{7a11}
\zeta_m = \sum_i \alpha_i \theta_i \circ a_i,  \text{ for some affine function $a_i$} 
\end{equation}
that we choose now. Call $c_i$ the antisymmetry center of the the wavelet $\theta_i$; 
we chose $a_i$ so that 
\begin{equation} \label{7a12}
a_i \circ F_{m-1}(c_i) = c_i \ \text{ and } \  (a_i \circ F_{m-1})'(c_i) = 1.
\end{equation}

Set $G_i = a_i \circ F_{m-1}$; since 
\begin{equation} \label{7a13}
\zeta_m \circ F_{m-1} = \sum_i \alpha_i \theta_i \circ G_i,  
\end{equation}
at least the centers of the pieces will be preserved. We shall prove by induction, 
among all, that 
\begin{equation} \label{7b14}
\text{$F_{m-1}$ is $(1+\alpha)$-bilipschitz,}
\end{equation}
with an $\alpha > 0$ that is as small as we want. This means that once we are done with the $\xi_m$ and $\zeta_m$,
we will have to check that \eqref{7b14} holds for $F_m$. Of course this is trivial when $m=1$ (and $F_0 = I$).

Because of \eqref{7b14}, $a_i$ also is $(1+\alpha)$-bilipschitz, and then 
the pieces $\lambda_i \theta_i \circ a_i $ that compose $\zeta_m$ are still supported in intervals $\wh J_i$ 
centered at the same points as the $J_i$, and barely larger than the $J_i$. In particular the $\wh J_i$ still live at positive distances 
(larger than $C^{-1} 2^{-n(m)}$) from each other. 
Because of this, and since composing with an affine mapping preserves the antisymmetry of each piece,
$\zeta_m$ satisfies \eqref{7a9}. Of course this was the point of taking $a_i$ affine.
Also, when we compare the formulas \eqref{7a10} and \eqref{7a11}, we see that 
\begin{equation} \label{7b15}
\frac12 N_j(\xi_m) \leq N_j(\zeta_m) \leq 2 N_j(\xi_m)
\ \text{ for } 0 \leq j \leq k,
\end{equation}
again, because $a_i$ is $(1 + \alpha)$-bilipschitz for some 
small $\alpha$. We will use the specific choice of \eqref{7a12} to make some terms smaller.
For instance, it follows from \eqref{7a13} that for $x\in \wh J_i$, 
\begin{equation} \label{7a14}
|G_i(x) - x| = |(a_i \circ F_{m-1})(x) - x| \leq |\wh J_i|^2 ||(a_i \circ F_{m-1})''||_{L^\infty(\wh J_i)}
\leq C 2^{-2n(m)} N_2(F_{m-1})
\end{equation}
where $n(m)$ still denotes the current generation. We return to $\delta_m$. Recall that
\begin{equation} \label{7a15}
\delta_m = \sum_i \alpha_i (\theta_i \circ G_i - \theta_i)
\end{equation}
by \eqref{7a8}, \eqref{7a10}, and \eqref{7a13}. We want to estimate its norm $N_k(\delta_m)$. 
First observe that by definition \eqref{7a4} of $f_j$ and iterations of \eqref{3e3} one can show 
$$
N_k(F_{m-1}) \leq \sum_{j \leq m-1} N_k(\zeta_j) + C \Big[\sum_{j \leq m-1} N_k(\zeta_j) \Big]^2 
$$
e.g., again by induction, using 
$N_k(F_m) \le N_k(\zeta_m) + N_k(\psi_{F_{m - 1}}) + CN_k(\zeta_m)N_k(\psi_{F_{m - 1}})$ (with $\psi_{F_{m - 1}} = F_{m - 1} - Id$). Then, by \eqref{7b15}, 
\begin{equation} \label{7a16}
N_k(F_{m-1}) \leq \sum_{j \leq m-1} N_k(\zeta_j) + C \Big[\sum_{j \leq m-1} N_k(\zeta_j) \Big]^2
\leq 2\sum_{j \leq m-1} N_k(\xi_j) + C \Big[\sum_{j \leq m-1} N_k(\xi_j) \Big]^2.
\end{equation}
We like the second terms because they are of order $2$. 
The affine mappings are harmless, so we immediately get that 
\begin{equation} \label{7a17}
N_k(G_i) \leq 2 N_k(F_{m-1}).
\end{equation}
Let us differentiate $k$ times the first piece $\theta_i \circ G_i(x)$, and of course, for a given $x$,  
restrict our attention to the interval $\wh J_i$ that contains $x$; if there is none, then $\delta_j=0$. 
The main term is $A(x) = G_i'(x)^k \theta_i^{(k)}(G_i(x))$, where we always differentiate $\theta_i$. But 
\begin{equation} \label{7a18}
|G_i'(x) - 1| \leq |x - c_i| ||G_i''||_\infty \leq C 2^{-n(m)} N_2(G_i) \leq C 2^{-n(m)} N_k(F_{m-1}),
\end{equation}
so we have
\begin{eqnarray} \label{7a19}
|A(x) - \theta_i^{(k)}(x)| 
&\leq& |G_i'(x)^k -1| \, |\theta_i^{(k)}(G_i(x))| + |\theta_i^{(k)}(G_i(x)) - \theta_i^{(k)}(x)|
\nn\\
&\leq& C 2^{-n(m)} N_k(F_{m-1}) |\theta_i^{(k)}(G_i(x))| + |\theta_i^{(k)}(G_i(x)) - \theta_i^{(k)}(x)|
\nn \\
&\leq& C 2^{-n(m)} N_k(F_{m-1}) |\theta_i^{(k)}(G_i(x))| + \|\theta_i^{(k+1)}\|_\infty 2^{-2n(m)} N_2(F_{m-1}),
\end{eqnarray}
where the second inequality holds thanks to \eqref{7b14}.
This is where we are about to pass the dangerous step, since we just introduced an extra derivative, as we said we would at the end of Section \ref{S4}. Fortunately, this derivative falls on a wavelet, and can be evaluated by homogeneity.
Note that this is the reason why we demanded one more derivative in \eqref{5a7}. Indeed, \eqref{5a10}
yields 
\begin{equation} \label{7a20}
\|\theta_i^{(k)}\|_\infty \sim 2^{-n(m)/2} 2^{n(m) k}
\ \text{ and } \ 
\|\theta_i^{(k+1)}\|_\infty \sim 2^{-n(m)/2} 2^{n(m)(k+1)},
\end{equation}
so 
\begin{equation} \label{7b23}
|A(x) - \theta_i^{(k)}(x)| \leq C 2^{-n(m)} N_k(F_{m-1}) \|\theta_i^{(k)}\|_\infty 
+ C\|\theta_i^{(k)}\|_\infty 2^{-n(m)} N_2(F_{m-1}),
\end{equation}
where we just lost half of the decay of the second term because we took this extra derivative. This gives
\begin{eqnarray} \label{7b24}
|\alpha_i| |A(x) - \theta_i^{(k)}(x)| &\leq& C 2^{-n(m)} N_k(F_{m-1}) \|\xi_m^{(k)}\|_\infty 
+ C\|\xi_m^{(k)}\|_\infty 2^{-n(m)} N_2(F_{m-1})
\nn\\
&\leq& C 2^{-n(m)} N_k(F_{m-1}) \|\xi_m^{(k)}\|_\infty = C 2^{-n(m)} N_k(F_{m-1}) N_k(\xi_m)
\end{eqnarray}
by \eqref{7a10} and the fact that the $\theta_i$ in the sum have disjoint supports.
This was for the main term coming from $A$, but we also need to estimate the other terms
in $(\theta_i \circ G_i)^{(k)}$, where not all the derivatives hit $\theta_i$. 
These terms are easier because they contain at least one second derivative of $G_i$ 
(which is at least of order $1$, as we saw in \eqref{7a18}),
and are also of order $1$ from the $\theta_i$ 
(because of the $\alpha_i$ and the comparison with $\xi_{m}$).
More precisely, these terms are dominated by $C N_k(\xi_{m}) N_k(F_{m-1})$. Altogether,
\begin{equation} \label{7a22}
N_k(\delta_m) \leq C N_k(\xi_{m}) N_k(F_{m-1}),
\end{equation}
which will be (quadratic) enough for us. In addition, observe that 
\begin{equation} \label{7b26}
N_k(\xi_{m}) \leq C N_k(\psi_2) 
\end{equation}
by (the proof of) \eqref{5a16}.
We are now ready to sum over $m$. We are not afraid to lose a small polynomial in $N$, 
the total number of generations, and anyway it seems that we have to. By \eqref{7a16} and \eqref{7b26}
\begin{equation} \label{7b27}
N_k(F_{m-1}) \leq C (1+m) N_k(\psi_2) + C (1+m)^2 N_k(\psi_2)^2
\end{equation}
so \eqref{7a22} yields 
\begin{equation} \label{7a23}
N_k(\delta_m) \leq C(1+m) N_k(\psi_2)^2 + C(1+m)^2 N_k(\psi_2)^3
\end{equation}
and for $\delta = \sum \delta_m$ from \eqref{7a7}, 
\begin{equation} \label{7a23bis}
N_k(\delta) \leq C (1+N)^2  N_k(\psi_2)^2 + C(1+N)^3 N_k(\psi_2)^3.
\end{equation}
This is the estimate that we wanted. Notice that since $N_k(\zeta_{m}) \leq C N_k(\psi_2)$
by \eqref{7b26} and \eqref{7b15}, we also have $N_k(f_m) = N_k(\zeta_m) \leq C N_k(\psi_2)$
(see \eqref{7a4}), so by the composition rule (or \eqref{3e3}), we also have 
\begin{equation} \label{7b30}
N_k(F_{m}) \leq C (2+m) N_k(\psi_2) + C (2+m)^2 N_k(\psi_2)^2,
\end{equation}
and \eqref{7b14} holds for $m$ if $\varepsilon$ is small enough for the quadratic term to be small.
For Theorem~\ref{t1a9}, this is the rule (we are allowed to choose $\varepsilon$ depending on $M$, or, equivalently, $N$); for Theorem \ref{t1}, this is also how we will choose $\varepsilon$ in the next section.

\subsection{End of proof of Theorem \ref{t1a9}}

At this point we are able to choose a constructible final map $F_{j_N}$, whose symbol
$\psi_{F_{j_N}} = F_{j_N}-Id$ is almost the same as the main part $\xi_N$ of our initial target $\psi_2$
(see \eqref{7a1} and \eqref{7a7}). 

Let us first check that $F_{j_N}$ satisfies all the requirements for Theorem~\ref{t1a9}, and then later 
we will return to Theorem \ref{t1}, choose a new target $\psi_3$ (the last one of our main construction)
and finish the first round of estimate before we iterate. 

For Theorem~\ref{t1a9}, we need to take a small notational detour, 
after which we will see that we essentially have all the direct estimates. 
We started the approximation process with the construction of $\psi_1$ and $\psi_2$,
and we targeted $f_\psi$ by a first composition of mappings $g_2 = f_{proj} \circ f_{cur}$,
where each piece is the composition of at most $C$ mappings of $\Psi(I)$. 
The number of mappings depends on $k$ through, for instance, the number of wavelets $\varphi$ in our basis of $V_0$, 
and even their regularity, but this is all right. Thus, in our description of the $F_M$ in Theorem~\ref{t1a9}, 
which we shall now call $\wt F_M$ to avoid confusion with the notation of this section,
the first mappings are here to reconstruct $f_{proj} \circ f_{c}$. So at some point we reach
\begin{equation} \label{7b31}
\wt F_{M_2} = f_{proj} \circ f_{c},
\end{equation}
and then we intend to continue the composition with the $f_m$ and $F_m$ of this section. 
The definition of the new target $f_{\psi_2}$ at this point is that 
\begin{equation} \label{7b32}
f = f_\psi = f_{\psi_2} \circ f_{proj} \circ f_{c},
\end{equation}
where $f_{\psi_2} = I + \psi_2$ as usual; see above \eqref{6a1}, 
where we say that $f_{\psi_1} = f_{\psi_2} \circ f_{\psi_{proj}}$, and \eqref{2e4} that says that 
$f_\psi = f_{\psi_1} \circ f_c$. Thus 
\begin{equation} \label{7b33}
\wt F_{M_2} = f_{\psi_2}^{-1} \circ f.
\end{equation}
From $\wt F_{M_2}$, we compose with the mappings $f_m$ that we constructed in this section, and 
set 
\begin{equation} \label{7b34}
\wt F_{M} = F_m \circ \wt F_{M_2}, \text{ with $m = M-M_2$}.
\end{equation}
The total number of mappings is $C + j_N \leq C_0 N$, as promised. 

We need to check that all our partial compositions $\wt F_M$ satisfy \eqref{1a16}.
For the first ones, up to $F_{M_2}$, this is by our proof that 
\begin{equation} \label{7b35}
N_k(\psi_2) \leq C N_k(\psi_1) \leq C N_k(\psi) + C N_0(\psi),
\end{equation}
by \eqref{6a17} and \eqref{2e9}, and where $C N_0(\psi)$ is needed when we do not assume that $\psi$ is 
compactly supported in $I$. 
We say the proof because we also have to estimate the intermediate mappings
that compose $f_c$ and $f_{proj}$.

For the next ones, observe that $\wt F_M = F_m \circ \wt F_{M_2}$, and that we already have the 
desired control on $\wt F_{M_2}$. We apply again the composition rule \eqref{3e3}, with 
$\varphi = F_{M_2}-Id$ and $\psi = F_m -Id$, and see that it is enough to prove (the analogue of) 
\eqref{1a16} for the $F_m$. This is fortunate, because \eqref{7b30} and \eqref{7b31} give exactly 
the required estimate (provided that $\varepsilon$ is small enough, depending on $N$).

We are left with \eqref{1a17} to check. 
Call $\wt F = \wt F_{C_0 N} = F_{j_N} \circ \wt F_{M_2} = F_{j_N} \circ f_{\psi_2}^{-1} \circ f$ 
the last mapping. We need to estimate the quantity
\begin{equation} \label{7b36}
\Delta : = \|\wt F - f \|_{L^\infty} = \|F_{j_N} \circ f_{\psi_2}^{-1} \circ f - f \|_{L^\infty}
= \|F_{j_N} - f_{\psi_2}\|_{L^\infty}
\end{equation}
(composing with $f^{-1} \circ f_{\psi_2}$ does not change the sup norm). Observe that
\begin{equation} \label{7b37}
F_{j_N} - f_{\psi_2} = F_{j_N} - Id - \psi_2 =  \delta + \xi_N - \psi_2 = \delta -\eta - \xi_\infty
\end{equation}
by \eqref{7a7} and \eqref{7a1}, so we have three small terms to estimate. 
By \eqref{7a1} and \eqref{6a16}, $\eta = P_0(\psi_2)$
and 
\begin{equation} \label{7b38}
N_0(\eta) \leq N_k(\eta) \leq C N_k(\psi_1)^{3/2} \leq C (N_k(\psi)+N_0(\psi))^{3/2},
\end{equation}
which is good enough for \eqref{1a17} if $\varepsilon$ is small enough. Next
$\xi = \psi_2 - \eta = \psi_2 - P_0(\psi_2)$
and $\xi_\infty = \xi - P_{N+1}(\xi) = \sum_{j > N+1} Q_j(\xi) 
= \sum_{j > N+1} Q_j(\psi_2)$ because $Q_j(P_0(\psi_2)) = 0$, so \eqref{5a22} says that 
\begin{equation} \label{7b39}
N_0(\xi_\infty) \leq \sum_{j > N+1} N_0(Q_j(\psi_2)) \leq C 2^{-kN} N_k(\psi_2) 
\leq C 2^{-kN} (N_k(\psi)+N_0(\psi)),
\end{equation}
as needed. Finally by \eqref{7a23bis} 
\begin{eqnarray} \label{7b40}
N_0(\delta) &\leq& N_k(\delta) \leq C (1+N)^2 N_k(\psi_2)^2 + C (1+N)^3 N_k(\psi_2)^3
\nn\\
&\leq& C (1+N)^2 (N_k(\psi)+N_0(\psi))^2 + C (1+N)^3 (N_k(\psi)+N_0(\psi))^3
\end{eqnarray}
Altogether,
\begin{eqnarray} \label{7b41} 
\Delta &\leq&  C 2^{-kN} (N_k(\psi)+N_0(\psi)) + C (N_k(\psi)+N_0(\psi))^{3/2} 
\nn\\
&\,&\hskip20pt
+C (1+N)^2 (N_k(\psi)+N_0(\psi))^2 + C (1+N)^3 (N_k(\psi)+N_0(\psi))^3,
\end{eqnarray}
in accordance to \eqref{1a17}.
This complete the proof of Theorem \ref{t1a9}. \qed

\section{We iterate at smaller scales} 
\label{S9}

Let us return to Theorem \ref{t1} and summarise how we complete our first round of approximation, before we do anything else. We have reached an interesting constructible function $F_{j_N}$, and we need to 
compare it to $Id +\psi_2$ to compute our next target $\psi_3$. That is, we define $\psi_3$
as usual by
\begin{equation} \label{7b42}
f_{\psi_2} = f_{\psi_3} \circ F_{j_N},
\end{equation}
or $f_{\psi_3} = F_{j_N} \circ f_{\psi_2}^{-1}$, and we estimate $\psi_3$.
First of all, we have the estimate \eqref{7b30} on 
$N_k(F_{j_N})$, and the same one for its inverse (see \eqref{3e10}), from which we deduce (by \eqref{3e9}) that 
\begin{eqnarray} \label{7b43}
N_k(\psi_3) &\leq& C (1+N) N_k(\psi_2) + C (1+N)^2 N_k(\psi_2)^2
\nn\\
&\leq& C (1+N) [N_k(\psi)+N_0(\psi)] + C (1+N)^2 [N_k(\psi)+N_0(\psi)]^2.
\end{eqnarray}
This is not much worse than what we started with, but more importantly we have much better
$L^\infty$ bounds on $N_0(\psi_3)$. Indeed, 
\begin{equation} \label{7b44}
\psi_3(x) = f_{\psi_3}(x)-x = F_{j_N} \circ f_{\psi_2}^{-1}(x) - x,
\end{equation}
hence by \eqref{7b36} and \eqref{7b41}, 
\begin{eqnarray} \label{7b45} 
\|\psi_3\|_\infty = \|F_{j_N} - f_{\psi_2}\|_\infty = \Delta 
&\leq&  C 2^{-kN} (N_k(\psi) +N_0(\psi)) + C (N_k(\psi)+N_0(\psi))^{3/2} 
\nn\\
&\,&\hskip1pt
+C (1+N)^2 (N_k(\psi)+N_0(\psi))^2 + C (1+N)^3 (N_k(\psi)+N_0(\psi))^3.
\end{eqnarray}
This concludes happily the first round of our construction.

~\

Let us review again the situation before we iterate. We started with a small target $\psi = \psi_0 = f - Id$, assumed that 
\begin{equation} \label{8a1}
a_0 := N_0(\psi_0) + N_k(\psi_0) \leq \varepsilon
\end{equation}
(in fact, it was the maximum, but the sum will be more convenient and there is no real difference), and showed that
we can replace $\psi_0$ with a new target $\psi_3$, for which we proved the bounds 
\eqref{7b43} and \eqref{7b45} for $N_k(\psi_3)$ and $N_0(\psi_3)$.

We want to apply the construction of Section \ref{S8} to $\psi_3$, but at a different scale 
$\ell$, small enough, which we shall choose later.
The rescaled functions $\ol f$ and $\ol \psi$ associated to $f$ and $\psi$ are defined by 
\begin{equation} \label{8a2}
\ol f(x) = \ell^{-1 }f(\ell^{-1}x)  \ \text{ and } \ \ol \psi(x) = \ell^{-1 }\psi(\ell^{-1}x);
\end{equation}
we choose the normalization that preserves the Lipschitz constants of $f$ and $\psi$, which is reasonable 
because $Id$ plays a special role and ought to be preserved. Then the other norms scale like this:
\begin{equation} \label{8a3}
N_j(\ol \psi) = \ell^{j-1} N_j(\psi) \ \text{ for } 0 \leq j \leq k.
\end{equation}
With these conventions, the rescaled norms for $\psi_3$ are
\begin{eqnarray} \label{8a4}
N_k(\ol \psi_3) &=&  \ell^{k-1} N_k(\psi_3)
\leq C \ell^{k-1} (1+N) [N_k(\psi)+N_0(\psi)] + C \ell^{k-1}(1+N)^2 [N_k(\psi)+N_0(\psi)]^2
\nn\\ 
&\,&\hskip60pt
\leq C \ell^{k-1} (1+N) a_0 + C \ell^{k-1} (1+N)^2 \varepsilon a_0
\end{eqnarray}
by \eqref{7b43} and \eqref{8a1}, and 
\begin{eqnarray} \label{8a5}
N_0(\ol \psi_3) &=&  \ell^{-1} N_0(\psi_3)
\leq C \ell^{-1} 2^{-kN} a_0+ C \ell^{-1} a_0^{3/2}
+C \ell^{-1} (1+N)^2 a_0^2 + C \ell^{-1} (1+N)^3 a_0^3
\nn\\
&\leq& C \ell^{-1} a_0 \big\{ 2^{-kN} + \varepsilon^{1/2} + (1+N)^2 \varepsilon + (1+N)^3 \varepsilon^2\big\}
\end{eqnarray}
by \eqref{7b45} and \eqref{8a1}. We want to make sure that 
\begin{equation} \label{8a6}
a_1 : = N_k(\ol \psi_3) + N_0(\ol \psi_3) \leq \tau a_0 \leq \tau \varepsilon,
\end{equation}
where we added a convergence factor $\tau < 1$ for later convenience, because this way we 
can apply the construction of Section \ref{S8} (as it is, at scale $1$) to $\ol\psi_3$. And of course we will 
iterate. So it will be enough to make sure that 
\begin{equation} \label{8a7}
C \ell^{k-1} (1+N)  + C \ell^{k-1} (1+N)^2 \varepsilon \leq \tau/2
\end{equation}
and 
\begin{equation} \label{8a8}
C \ell^{-1} \big\{ 2^{-kN} + \varepsilon^{1/2} + (1+N)^2 \varepsilon + (1+N)^3 \varepsilon^2\big\}
 \leq \tau/2.
\end{equation}
Let us give our constraints on the constants $N$, $\ell$, $\tau$, and $\varepsilon$;
maybe more constraints will show up in the next section, but they will be compatible.
We intend to make sure that 
\begin{equation} \label{8a9}
C(1+N) \ell^{k-1} \leq \frac{\tau}{4}   \text{ and } C 2^{-kN} \leq \frac{\tau \ell}{4} ,
\end{equation}
with $C$ as above, and for this we decide that 
\begin{equation} \label{8a10}
\tau = \ell^{k-\frac54}
\end{equation}
(in fact any power smaller than $k-1$ and close enough will do the job here), and 
$N$ as small as possible, but such that 
\begin{equation} \label{8a11}
C 2^{-kN} \leq \frac{\tau \ell}{4} = \frac{\ell^{k-\frac14}}{4}.
\end{equation}
The precise power would not matter much, the point is that $N \leq C \log(1/\ell)$.
Then, for $\ell$ small enough, the first part of \eqref{8a9} holds as well.

Observe here that, given our choice of the parameter $\tau$, the relation between $N$ and $l$ is something very natural: the condition $N \leq C \log(1/\ell)$ says that $l \le 2^{-N}$, so, in other words, we start iterating necessarily from a scale smaller than the one we worked on at the previous iteration of our algorithm. The condition $C 2^{-kN} \leq \frac{\tau \ell}{4}$ says that that scale should not be too small at the same time.

Once the constants $N, l, \tau$ are chosen, we choose $\varepsilon$ so small (depending on $\ell$)
that \eqref{8a7} and \eqref{8a8} hold (or we make it smaller than our previous choices).

With this choice, we see that the next function $\ol\psi_3$ satisfies 
the requirement \eqref{8a1} (and even better), 
so we can associate a new target $\wt \psi_6$ to it, and the normalised function 
$\ol \psi_6(x) = \ell^{-1 }\wt\psi_6(x/\ell)$. We keep the index $6$ to avoid changing the 
name of $\psi_3$, but we will not discuss the intermediate functions labelled $4$ and $5$.
The function $\wt\psi_6$ is somewhere between plain and normalised, 
in the sense that if we had performed the construction of Section \ref{S8}, starting from $\psi_3$ and at the scale $\ell$, 
we would have obtained a function $\psi_6$ such that $\psi_6(x) = \ell \wt\psi_6(\ell x)$.
Thus the relation between plain and normalised is
\begin{equation} \label{8a12}
\ol \psi_6(x) = \ell^{-2 }\psi_6(\ell^{-2}x). 
\end{equation}
The estimates of Section \ref{S8} yield, as for \eqref{8a6}
\begin{equation} \label{8a13}
a_2 : = N_k(\ol \psi_6) + N_0(\ol \psi_6) \leq \tau a_1 \leq \tau^2 a_0,
\end{equation}
and we are in position to apply Section \ref{S8} again, at scale $\ell^2$, and so on. We get
sequences of non-normalised targets $\psi_{3j}$ and normalised one $\ol \psi_{3j}$, 
with the relations
\begin{equation} \label{8a14}
\ol \psi_{3j}(x) = \ell^{-j }\psi_{3j}(\ell^{-j}x), 
\end{equation}
\begin{equation} \label{8a15}
a_j : = N_k(\ol \psi_{3j}) + N_0(\ol \psi_{3j}) \leq \tau^j a_0.
\end{equation}

Let us observe that with such choices, our successive targets $\psi_{3j}$ lie uniformly
in $C^{k-1}$, and even with decaying norms, because 
\begin{eqnarray} \label{8a16}
N_{k-1}(\psi_{3j}) &=& \ell^{(1-(k-1))j} N_{k-1}(\ol\psi_{3j})
\leq \ell^{(2-k)j} (N_{k}(\ol\psi_{3j})+N_0(\ol\psi_{3j}))  
\nn\\
&=& \ell^{(2-k)j} a_j \leq \ell^{(2-k)j} \tau^j a_0 
=\ell^{(2-k)j} \ell^{(k-\frac54)j} a_0 = \ell^{3j/4} a_0
\end{eqnarray}
by \eqref{8a3}, \eqref{8a15}, and \eqref{8a10}. And with the same proof,
\begin{equation} \label{8a17}
N_{0}(\psi_{3j}) \leq \ell^{(k - \frac{1}4) j} a_0.
\end{equation}

This is good to know, because it says that our finite compositions (call them $F_j$)
of constructible mappings converge in $C^{k-1}$ to $f$. Indeed, the usual target relation is 
\begin{equation} \label{8a18}
f = f_\psi = f_{\psi_{3j}} \circ F_j  \ \text{ or equivalently } \   F_j = f_{\psi_{3j}}^{-1} \circ f
\end{equation}
(compare with \eqref{7b32} or similar identities), and $f_{\psi_{3j}}^{-1}$ converges to $Id$
in $C^{k-1}$.

\medskip

Unfortunately, for the geometric control on the curves that will be done in the next section,
and even to control some intermediate mappings, 
we get a worse estimate that does not warranty a uniform control in $C^{k-1}$.

\section{We construct the green and red curves} 
\label{S10}

So far we just constructed a sequence of constructible mappings 
$\{ f_n = Id + \psi_n \}$ so that 
$F_q = f_{q} \circ \ldots \circ f_1$ converges to $f$, as in \eqref{8a18}.
Each $f_n$ was constructed in a generation $j(n)$, and we proved good bounds for the 
$f_n$ and the $F_q$. We allowed less than $CN$ mappings 
for each given generation $j$; this is all right, because now $N$ is fixed and chosen.

But now we want to construct curves associated to these mappings, and pile them to get a global
set of red and green curves. We thus cut $J$ into intervals $J_n$, so that 
($\sum_n |J_n| = |J|$ and) 
\begin{equation} \label{9a1}
|J_n| \sim (CN)^{-1}\ell^{j(n)} |J| \sim \ell^{j(n)} |J|.
\end{equation}
It could sound appealing to take a slower rate of decay for the $|J_n|$, and indeed it would help make 
the $t$-derivatives below smaller, but it does not seem to help getting better estimates globally,
because we also need to control $x$-derivatives. So we don't try. 
We put the $J_n$ in the same order that the $f_n$ were constructed, i.e., large intervals at the bottom (if we see them as vertical). 
To feel more comfortable, we may even add a small buffer interval between consecutive $J_n$,
where the green curves will be horizontal segments and the red ones vertical segments, but this is not
really needed as our curves will glue nicely.

Each $f_n$ is of the form $Id + \psi_n$ (a different numbering than in the previous sections),
where $\psi_n \in \Psi(I)$ is constructible and a sum of antisymmetric mappings with 
disjoint supports. We have seen in Lemma \ref{l1a5} how to construct, in $R_n = I \times J_n$,
green and red curves associated to $\psi_n$, and now we put all these curves together, i.e., take 
all the green curves constructed for the $\psi_n$, and for the red curve, concatenate (with a nice gluing)
all the red curves that we have constructed. This is how we define the curves for Theorem \ref{t1}; we now need to check that they satisfy the conclusions of the theorem.

For the green curves, they foliate $R$ because the green curves for $\psi_n$ foliate $R_n$, and they satisfy 
\eqref{1a1} by construction.

For the red curves, they also foliate $R$ for the same reason, and \eqref{1a2} follows from the fact
that the red curve that starts from $(x,0)$ goes to the successive $(F_q(x), t_q)$, where $t_q$ is the final endpoint of $J_q$, and we saw that the $F_q$ converge (even in $C^{k-1}$) to $f$. 

But we also have to check that our curves are globally $C^m$, where $m$ is the integer part of $k/2$.
Recall that this means that the functions $r_{x}$ (whose graph is $R_{x}$) and $g_{t}$ (whose graph is $G_{t}$) 
are all $C^m$ functions of $(x,t)$. By the locality of $C^m$, it will be enough to check this on each rectangle $R$
of the construction, with uniform estimates, so that no bad surprise happens at the top of the construction (i.e., the place where we accumulate the rectangles corresponding to the later generations in the algorithm).

Maybe we should recall here that for generations $j(n) \geq 2$, we actually work in parallel on many horizontal intervals (call them $I_{n,i}$)
of roughly the same size $\ell^{j(n)}|I|$, with different intermediate targets $\psi_{n,i}$, and our estimates are independent. 
This is made easier by the fact that we can use the curtain trick
to have disjoint supports $I_{n,i}$ at each step. We will pretend below that we work with a single $I_{n,i}$, but not right now. 

So we consider the curves in $R_n$, associated to functions $\psi_n$, and let us go directly to the ones that 
give the worst estimates, which are the ones coming from the commutators of Section \ref{S5}. 
Recall that we grouped together four maps $f=f_t$, $g=g_t$, and their inverse, to form a commutator 
$R_t$. We did this independently on various disjoint intervals $I_r$, $1 \leq r \leq \wt n$ (were the notation for $r$ and $\wt n$ was $j$ and $n$ respectively in Section \ref{S7}, which we now change to avoid the confusion, and the only point in $\wt n$ is that it is bounded), to project on the orthogonal of $V_0$, and we computed the coefficients 
$t = t_r$, from other coefficients $\mu_r$, with the rule $\mu_r = \frac{t_r^2}{2} \, {\rm sign}(t_r)$
(above \eqref{6a8}), that we chose to solve a system of equations. Finally the $\mu_r$ satisfy 
\eqref{6b15}, and 
\begin{equation} \label{9a2}
|t_r| \leq C \sup_r |\mu_r|^{1/2} \leq C N_0(\psi_1)^{1/2} \leq C a_0^{1/2},
\end{equation}
where $a_0$ is given by \eqref{8a1}. 
And thus 
\begin{equation} \label{9a3}
N_k(f_{t_r}) + N_0(f_{t_r}) \leq C a_0^{1/2},
\end{equation}
because $f_{t_r}$ always has a small support, and similarly for the other three functions. 
For further generations $j$, we would still have 
\begin{equation} \label{9a4}
N_k(\ol f_{t_r}) + N_0(\ol f_{t_r}) \leq C a_j^{1/2} \leq C \tau^{j/2} a_0^{1/2} = C \ell^{j(\frac{k}2 -\frac58)} a_0^{1/2},
\end{equation}
by \eqref{8a15} and \eqref{8a10}, and similarly for the other three functions.
Now the analogue of \eqref{8a16} is different, because of the square root:
\begin{equation} \label{9a5}
N_{m}(f_{t_r}) = \ell^{(1-m)j} N_{m}(\ol f_{t_r})
\leq \ell^{(1-m)j} \big\{N_{0}(\ol f_{t_r}) + N_{k}(\ol f_{t_r}) \big\}
\leq C \ell^{j(\frac{k}2 +\frac38 -m)} a_0^{1/2}.
\end{equation}
The exponent is positive when $m < \frac{k}{2} +\frac{3}{8}$ which, since we decided to consider
only integers, means that $m$ is at most the integer part of $\frac{k}2$ (call this $m_0$
for the moment), as in the statement of Theorem \ref{t1}.
Thus, the partial regularity result at the end of Section \ref{S9} was a little deceptive, in the sense that 
we have intermediate values of $t \in J$ where we can only prove that our partial targets $f_{t_r}$
(and, we will see soon, our curves $G_t$) are uniformly $C^m$ when $m \leq m_0$. 
See below for some more details and verifications.

This was the worse case: for the other rectangles $R_n$ coming from the curtain trick, or the approximation of wavelets, the same computation (without the square root) yields that our intermediate targets $\psi$ are uniformly $C^{k-1}$ on $R_k$.

\ms
We now return to the geometry and estimates on curves, and 
we need to check that for $m \leq m_0$, our functions $g_t$ and $r_x$ are uniformly $C^m$ on the rectangles $R_n = I \times J_n$ ($J_n$ is as in \eqref{9a1}). Recall that, according to our construction, during the step corresponding to the rectangle $R_n$, we treat a function $\psi$ which is a sum of many terms with disjoint supports $I_n$ of size roughly $\ell^{j(n)}|I|$: we called them $I_{n, i}$ on the previous page. We now pretend that, instead of working on $R_n$, we work on $I_n \times J_n$. 
Again, the worst case can arise from the commutators, and we can take $I_n$ to be one of the intervals that come from the special mappings in Section \ref{S5}, but these are still of the same size $\ell^{j(n)}|I|$. 
Then the construction for $g_t$ and $r_x$ is given in Lemma \ref{l1a5}. With $I = I_n$ and $J = J_n$, it yields
\begin{equation} \label{10b6}
||\nabla_{x,t}^m \, g_t(x)||_\infty + ||\nabla_{x,t}^m\, r_x(x)||_\infty \leq C \left(1 + \frac{|I|}{|J|}\right)^{(m + 1)}\sum_{i \leq m} N_i(f) 
\, \text{ for } 1 \leq m \leq m_0,
\end{equation}
where $f$ is one of our $f_{t_r}$ from above. 
We have $N_{i}(f) \leq C \ell^{j(n)(\frac{k}{2} +\frac{3}{8} -i)} a_0^{1/2}$ by \eqref{9a5}, so
\begin{equation} \label{10b7}
||\nabla_{x,t}^m \, g_t(x)||_\infty + ||\nabla_{x,t}^m\, r_x(x)||_\infty \leq C \left(1 + \frac{|I|}{|J|}\right)^{(m + 1)} \sum_{i \leq m} \ell^{j(n)(\frac{k}{2} +\frac{3}{8} - i)}  a_0^{1/2}
\, \text{ for } 1 \leq m \leq m_0. 
\end{equation}
As before, the exponents in the sum are positive, so we get the desired uniform estimate, plus a little bit of decay in $n$
to control the top. 
This completes our proof of Theorem \ref{t1}.
\qed

\section{One case where we can obtain a double control} 
\label{S11}

We present in this section a simple case where Theorem \ref{t1} allows us to solve the 
double control problem evoked in Remark \ref{r1a4}.

As for the theorem, we give ourselves a rectangle $R = I \times J$ and a function $f : I \to I$
that satisfies \eqref{1a4}. But we also give ourselves a function $g : J \to J$, that satisfies 
the analogue of \eqref{1a4} for $J$. We ask for an analogue of Theorem \ref{t1}, where we would replace 
\eqref{1a1} by a more general control constraint, namely 
\begin{equation}\label{10a1}
\text{ For $t\in J$, the green curve $G_t$ that starts at $(a, t)$ ends at $(b, g(t))$}.
\end{equation}

In general we do not know how to do this, but there is an easy special case, when we can use 
Theorem \ref{t1} to solve this double control problem. This is when we can find an interior point $x_0$
of $I$, and an interior point $y_0$ of $J$, such that 
\begin{equation} \label{10a2}
\psi = f-I \text{ has a zero of order $k$ at $x_0$,} 
\end{equation}
and 
\begin{equation} \label{10a3}
\varphi = g-I \text{ has a zero of order $k$ at $y_0$.} 
\end{equation}

Note, before we construct a solution for this case, that it contains the case when $\psi$ and $\varphi$
are supported in strict subintervals $I'$ and $J'$ of $I$ and $J$ respectively. 
But in both cases, the estimates for the solution deteriorate when $x_0$ or $y_0$ gets close 
to the boundary of $I$ or $J$, and when the length of $I'$ or $J'$ gets close to the length of the 
corresponding $I$ or $J$.

So let us construct a solution. Use $x_0$ and $y_0$ to split $I$ and $J$ as 
$I = I_1 \cup I_2 = [a, x_0] \cup [x_0,b]$ and $J = J_1 \cup J_2= [c, y_0] \cup [y_0,d]$.
For $i = 1,2$, call $f_i$ the restriction of $f$ to $I_i$, and $g_i$ the restriction of $g$ to $J_i$.

Then solve the problem of Theorem \ref{t1} in $I_1 \times J_2$ with data $f_1$, and 
in $I_2 \times J_1$ with data $f_2$. This makes sense because $f_i$ 
is an increasing bijection of $I_i$.
Also solve the problem of Theorem \ref{t1}, with the two coordinates 
exchanged, on $I_1 \times J_1$ with data $g_1$, and on $J_1 \times J_2$ 
with data $g_2$.
Again this is possible because $g_i$ is an increasing bijection on $J_i$.
This gives four sets of red and green curves, that we can nicely glue to each other (because we were
careful in \eqref{10a2} and \eqref{10a3} to get vanishing of order $k$). We get  
red and green curves that cover the whole $R$, and now we need to check \eqref{10a1} and \eqref{1a2}.
See Figure \ref{fig10}.

We prove \eqref{10a1}; \eqref{1a2} would be the same. 
Let $G_t$ be the green curve that starts at 
$(0,t)$, and first assume that $t \in J_1$. Then $G_t$ crosses $I_1 \times J_1$ first, and ends up
at $(x_0, g_1(t)) = (x_0, g(t))$. Then it crosses $I_2 \times J_1$, and it ends at the same altitude, i.e.
at $(b, g(t))$. If instead $t \in J_2$, then $G_t$ crosses $I_1 \times J_2$, and it ends up at
$(x_0, t)$. And then it crosses $I_2 \times J_2$, and ends up at $(b,g(t))$. See the picture. 
The verification of \eqref{1a2} is the same. \qed

\begin{figure}[!h]
  \centering 
 \includegraphics[width=7cm]{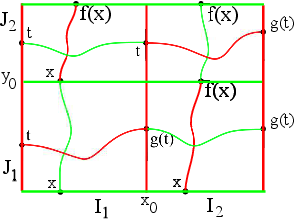} 
  \caption{Solution of the double control problem assuming two fixed points} 
\label{fig10}
\end{figure}

\section{Question \ref{q1a8} meets GPT: the exact decomposition of a small function into locally 
antisymmetric ones}
\label{S12}

In this section we provide a decomposition of the type \eqref{1a14} for the function $f$ in Theorem~\ref{t1}, 
which gives rise to a desired pair of foliations of curves. There are only three ingredients in this construction: 
Tsuboi's four-commutator theorem, something GPT called ``the paired-slot construction'', and our curtain trick, 
which turns out to be important here as well to rectify the mixture of the first two ingredients and get $f_i$ in \eqref{1a14} 
from $\Psi(I)$ and not from the class $\Psi(I')$ for a larger interval $I'$. 

We start with Tsuboi's four-commutator theorem presented in \cite{Ts}, about whose existence we were ignorant, and which was given to us by GPT. We keep the notation, including integer numbers, for different classes of regularity of functions as in Tsuboi's paper (not to be confused with our notation from the previous sections). 
We will also call support of a diffeomorphism $f$ the (usual) support of $\psi = f - Id$.
In addition, compositions will often be written as products (i.e., $fg$ instead of $f \circ g$).

\begin{thm}[Tsuboi, 1981]\label{tts}
Denote by $\mathcal{D}_K(n, c)$ the class of smooth diffeomorphisms $f$ of $[0, 1]$  ($f \in \mbox{Diff}^\infty([0, 1])$) supported inside $\left(\frac{1}{8}, \frac{7}{8}\right)$ such that $N_n(f - Id) < c$. Then, for any positive integer $m$, there exists a positive integer $n$ and a positive real number $c$ such that every element $f$ of $\mathcal{D}_K(n, c)$ can be written as a composition of four commutators; $f = [a_1, b_1][a_2, b_2][a_3, b_3][a_4, b_4]$, where each $a_i, b_i$ belongs to $\mbox{Diff}^\infty([0, 1])$ and satisfies 
\begin{equation}\label{ts_est}
    \supp(a_i), \supp(b_i) \subset (0, 1) \quad \mbox{and} \quad N_m(a_i - Id), N_m(b_i - Id) \lesssim N_n(f - Id)^{1/(m + 2)}.
\end{equation}
\end{thm}

Some remarks are in order. To start with, apparently this is not the first known result about a representation of a compactly supported function with a small norm by a composition of commutators: later, while browsing, we came upon a result of Mather \cite{Mat}. We could also use it instead of Tsuboi's theorem, but we really like that, unlike Mather's variant, Theorem \ref{tts} involves only four commutators: four is a very small number. Next, it is important to say that, even though Theorem \ref{tts} is stated for $C^\infty$ diffeomorphisms, it also works for $C^n$ diffeomorphisms, and produces $a_i, b_i$ which are $C^m$, as announced. The gap between $n$ and $m$ comes from a so-called circle-conjugacy theorem, and GTP also remarked that it apparently can be improved, if the circle-conjugacy theorem is replaced by its more modern version by Marmi, Moussa and Yoccoz, see \cite{MMY}. Last, we have mentioned in Subsection \ref{S2S1} that the estimate on norms of the functions involved in our approximation algorithm and the resulting curves in Theorem \ref{t1} are better than what is provided by a finite decomposition. Indeed, compare the estimate $\sim N_k(\psi)^{1/2}$ we have in Theorem \ref{t1} with \eqref{ts_est}. So, our curves from the previous sections are ``flatter''.

So far what it is not yet clear what the decomposition in Theorem \ref{tts} with commutators has to
do with decompositions of the type \eqref{1a14}, because the commutators $[a_i, b_i]$ are not necessarily
done with antisymmetric functions of $\Psi(I)$, so we don't know that they are constructible. 
But there is a beautiful simple trick, that GPT calls ``the paired-slot construction'', which allows one
to transform commutators like $[a_i, b_i]$ into commutators of functions of $\Psi(I)$. This is probably
a standard technique, but we were unfortunately unable to trace it to an exact source.

Here is the idea.
Suppose $a= Id + \psi_a$ and $b= Id + \psi_b$ are supported in an interval $I_0$, and select another interval
$I_1$ of the same length and disjoint from $I_0$. We call $I_1$ a slot (for $I_0$). 
Call $x_0$ the center of mass of $I_0 \cup I_1$, and call $\wt \psi_a$ the function which is 
supported on $I_0 \cup I_1$, equal to $\psi_a$ on $I_0$,
and antisymmetric with respect to $x_0$. Then set $\wt a = Id + \wt \psi_a$ on $\R$
(if $a$ was initially defined on $\R$, we changed its values on $I_1$). 
By construction, $\wt \psi_a \in \Psi_0(\wh I_1)$, where $\wh I_1$ is the convex hull of $I_0 \cup I_1$, so it constructible.
Do the same thing for $b$, but with another slot $I_2$ that does not meet $I_0 \cup I_1$. This
gives another constructible function $\wt \psi_b \in \Psi_0(\wh I_2)$, where $\wh I_2$ is the convex hull of $I_0 \cup I_2$. 
The beauty of the trick is that 
\begin{equation} \label{12b2}
[\wt a, \wt b] = [a, b],
\end{equation}
because one can check carefully that $[\wt a, \wt b] = Id$ both on $I_1$ and $I_2$, thanks to $\wt a = Id$ on $I_2$ and $\wt b = Id$ on $I_1$. 
So our initial, apparently non-constructible,
commutator $[a, b]$ becomes constructible in $I$, provided that $I$ contains the support $I_0$
of $a$ and $b$, plus the two disjoint slots $I_1$ and $I_2$.

\ms
We are now ready to decompose our initial mapping $f = Id + \psi$, where $\psi$ is supported
on $I = [0,1]$, as a composition of $3$ curtain mappings,
followed by $4$ ``slotted'' constructible commutators. Recall from Section \ref{S3} that we we replaced our initial
target $f = f_\psi$ by a new one, $f_{\psi_1} = f_{\psi} \circ f_c^{-1}$ (see \eqref{2e4}), where 
$\psi_c \in \Psi(I)$ and $\psi_1$ is now supported on $\left[\frac{1}{30}, 1 - \frac{1}{30}\right]$
(see \eqref{2e7}). This is because we did not try to optimize, and instead we can decide to kill a portion of size 
$\theta = 1/5$ near $0$ and also near $1$. We can do this with a mapping $f_c$ near $0$ that lives on an interval of size a little more than $2/5$, where we need some space to do cut-off functions, and some extra space because we know 
that $f_c$ will possibly increase or decrease the images of intervals; this is the same argument as in Section \ref{S3}. 
We do the same near $1$, on an interval that does not meet the previous one,
and get a better curtain mapping $f_c$ such that $\psi_1$ is now supported on an interval of length $3/5$ centered
at $1/2$. If the reader does not like this, they can use the same mapping as in Section \ref{S3}, and will just need
to compose more than $3$ functions like $f_c$. 

So we define $f_{\psi_1} = f_{\psi} \circ f_c^{-1}$ as above, with $\psi_1$ supported on an interval of length $3/5$ centered at $1/2$, and such that $N_k(\psi_1) \leq C N_k(\psi)$ as in \eqref{2e9}. We do this twice more, and write
$\psi = f_{\psi_3} \circ f_{c_2} \circ f_{c_1} \circ f_{c}$, with constructible mappings $f_c$
and $\psi_3$ which is supported in an interval of length $27/125$ centered at $1/2$ and such that 
$N_k(\psi_3) \leq C N_k(\psi)$. Then we use Tsuboi result to write 
$f_{\psi_3} = [a_1, b_1][a_2, b_2][a_3, b_3][a_4, b_4]$, with mappings $a_i$ and $b_i$ 
supported in an interval $I_0$ of length $(8/6)(27/125) = 36/125$ centered at $1/2$. This is short enough to find two 
to slots for $I_0$, one on each side of $I_0$, that are contained in $I$. Thus each commutator 
$[a_i, b_i]$ is constructible, by \eqref{12b2}, and we get the desired decomposition of $f$ into
$19$ pieces ($4$ pieces for each commutator).

\msi
{\bf Guy David}, Universit\'e Paris-Saclay, CNRS, Laboratoire de math\'ematiques d'Orsay, 91405 Orsay, France 
{\sl and}  Department of Mathematics ETH Z\"urich, 8092, Switzerland; 

guy.david@universite-paris-saclay.fr

\smallskip\noindent
{\bf Polina Perstneva}, Universit\'e de Gen\`eve, Section de Math\'ematique, Switzerland; 

polina.perstneva@unige.ch

\end{document}